\documentclass[12pt,leqno]{amsart}
\usepackage{graphicx}%
\usepackage{multirow}%
\usepackage{amsmath,amssymb,amsfonts,mathtools, esint}%
\usepackage{amsthm}%
\usepackage{mathrsfs}%
\usepackage[title]{appendix}%
\usepackage{xcolor}%
\usepackage{textcomp}%
\usepackage{manyfoot}%
\usepackage{booktabs}%
\usepackage{algorithm}%
\usepackage{algorithmicx}%
\usepackage{algpseudocode}%
\usepackage{listings}%
\usepackage{bbm}
\usepackage[pagebackref=true, colorlinks=true, citecolor=blue]{hyperref}
\theoremstyle{thmstyleone}%
\newtheorem{theorem}{Theorem}%
\newtheorem{proposition}[theorem]{Proposition}%
\newtheorem{remark}{Remark}%
\newtheorem{defn}{Definition}%
\newtheorem{thm}{Theorem}
\newtheorem{lemma}{Lemma}

\newtheorem{lem}[theorem]{Lemma}

\numberwithin{theorem}{section} 

\newcommand{\R}{\mathbb{R}}

\usepackage[foot]{amsaddr}
\usepackage[margin=1in]{geometry}

\newcommand{\cred}{\color{red}}

\newcommand{\ep}{\varepsilon}

\newcommand{\bn}{{\bf n}}
\newcommand{\bg}{{\bf g}}
\newcommand{\bu}{{\bf u}}
\newcommand{\bq}{{\bf q}}
\newcommand{\bw}{{\bf w}}
\newcommand{\bv}{{\bf v}}
\newcommand{\bD}{{\bf D}}
\newcommand{\bW}{{\bf W}}
\newcommand{\bH}{{\bf H}}
\newcommand{\bL}{{\bf L}}
\newcommand{\bQ}{{\bf Q}}
\newcommand{\bA}{{\bf A}}
\newcommand{\bC}{{\bf C}}

\newcommand{\sO}{\mathcal O}
\newcommand{\sW}{\mathscr{W}}
\newcommand{\sV}{\mathscr{V}}
\newcommand{\sD}{\mathscr{D}}

\newcommand{\sB}{\mathcal{B}}
\newcommand{\sL}{\mathcal{L}}

\newcommand{\bfeta}{{\boldsymbol{\eta}}}
\newcommand{\bfpsi}{\boldsymbol{\psi}}
\newcommand{\bfphi}{\boldsymbol{\phi}}
\newcommand{\bfvarphi}{\boldsymbol{\varphi}}
\newcommand{\bfxi}{\boldsymbol{\xi}}

\newcommand{\cblue}{\color{black}}

	\title[3D FSI with vector structure displacements]{Existence and Regularity Results for a Nonlinear Fluid-Structure Interaction Problem with Three-Dimensional Structural Displacement}

\author[ S. \v{C}ani\'{c}, B. Muha, K. Tawri ]{Sun\v{c}ica \v{C}ani\'{c}$^1$ and Boris Muha$^{2}$ and Krutika Tawri$^{1}$ }

\address{\newline	$^1$ Department of Mathematics, University of California Berkeley, CA, USA.
\newline $^2$ Department of Mathematics, University of Zagreb, Croatia.}

\begin{document}
	
	\begin{abstract}
In this paper we investigate a nonlinear fluid-structure interaction (FSI) problem involving the Navier-Stokes equations, which describe the flow of an incompressible, viscous fluid in a 3D domain interacting with a thin viscoelastic lateral wall. The wall's elastodynamics is modeled by a two-dimensional plate equation with fractional damping, accounting for displacement in all three directions. The system is nonlinearly coupled through kinematic and dynamic conditions imposed at the time-varying fluid-structure interface, whose location is not known a priori.

We establish three key results, particularly significant for FSI problems that account for vector displacements of thin structures. Specifically, we first establish a hidden spatial regularity for the structure displacement, which forms the basis for proving that self-contact of the structure will not occur within a finite time interval. Secondly, we demonstrate temporal regularity for both the structure and fluid velocities, which enables a new compactness result for three-dimensional structural displacements. Finally, building on these regularity results, we prove the existence of a local-in-time weak solution to the FSI problem. This is done through a constructive proof using time discretization via the Lie operator splitting method.

These results are significant because they address the well-known issues associated with the analysis of nonlinearly coupled FSI problems capturing vector displacements of elastic/viscoelastic structures in 3D, such as spatial and temporal regularity of weak solutions and their well-posedness. 
	\end{abstract}

	\maketitle
	
\section{Introduction}
We study a fluid-structure interaction (FSI) problem involving an incompressible, viscous fluid flowing within a three-dimensional domain, bounded by thin compliant lateral walls. The fluid dynamics is governed by the three-dimensional Navier-Stokes equations, while the structural dynamics is modeled by a linear viscoelastic plate equation incorporating fractional damping.

The interaction between the fluid and the structure is characterized by a fully coupled system, with kinematic and dynamic coupling conditions that enforce the continuity of velocities and contact forces at the dynamic fluid-structure interface. This coupling introduces a significant geometric nonlinearity to the problem, as the fluid domain's location is not known {\sl{a priori}} and is instead one of the unknowns in the problem.

The field of fluid-structure interaction (FSI) analysis has seen tremendous progress over the past two decades (see, e.g., \cite{bodnar2014fluid,kaltenbacher2018mathematical,Canic21} and references therein). In this paper, we focus on the interaction between fluid flow and a plate structure, so our brief literature review emphasizes the analysis of moving boundary FSI problems where the structural dynamics is described by lower-dimensional models.

Most existing works involving lower-dimensional models interacting with viscous incompressible fluids consider the case of scalar displacement, where the structure deforms only in a fixed direction typically normal to the reference configuration. The theory of weak solutions in this context is well developed, see \cite{CDEG, G08, MC13,Srdjan20,SebastianWSU,FSIHeat} and references therein. Strong solutions have also been studied in this context, as can be found in e.g. \cite{KT12,GH16,GrandmOntHillairetLeq19,MaityTakhashi21} and refences within.

The case of three-dimensional ($3D$) structural displacement, where the structure can deform in all three spatial directions (vector displacement), is less well-studied, with only a few works addressing weak solutions. In \cite{MCG20}, the authors investigated an FSI problem involving a $3D$ fluid flow interacting with a two-dimensional ($2D$) cylindrical shell supported by a mesh of elastic rods. They proved the existence of a weak solution under additional assumptions that ensured the structure's displacement remained Lipschitz continuous in space at all times. In the $2D$ fluid and $1D$ structure scenario, several results have been obtained. The local-in-time existence of weak solutions to FSI problems where $2D$ Navier-Stokes equations are coupled with $1D$ plate or shell equations via the Navier slip boundary condition is established in \cite{SunBorSlip}. More recently, \cite{KSS23} considered an FSI problem where the structure is described by a nonlinear beam equation with a term that penalizes compression, preventing domain degeneracy. Additionally, recent work \cite{grandmont2024existence} has established the existence of local-in-time strong solutions for an FSI problem where the structure is modeled as a linear plate.

{\cblue {To the best of our knowledge, the present work is the first to establish the existence of weak (finite energy) solutions for a moving boundary FSI problem where a $3D$ fluid is coupled with a $2D$ plate with $3D$ vector displacement. }}

The primary challenge in developing a theory for FSI problems involving structure equations accounting for 3D vector displacements is managing the difficulties associated with self-contact. Specifically, proving existence results requires ruling out fluid domain degeneracy, i.e., preventing self-contact of the structure over the time interval where the solution is defined. {\cblue{In particular, in the case of $3D$ displacement}}, the standard energy estimates do not provide sufficient regularity {\cblue{of}} the structure {\cblue{to analyze issues with self-contact}}. 

Another challenge in developing a theory for FSI problems with vector displacements and with the geometrically nonlinear coupling 
 is designing suitable compactness arguments for the fluid and structure velocities whose energy-based regularity estimates are insufficient to deduce compactness.

{\cblue{In this manuscript we address both of those challenges by proving two ``hidden'' regularity results for weak solutions of such problems. The first regularity result improves the spatial regularity
of structure displacement over the ``basic'' regularity provided by the energy estimate, and the second regularity result improves the temporal regularity of fluid and structure velocities over that provided by the energy estimates. The first is used in ensuring non-degeneracy of the fluid domain, while the second is used in establishing compactness arguments for the fluid and structure velocities in this class of nonlinear moving boundary problems.
Finally, building on these regularity results we prove the existence of a weak solution to a FSI involving 3D Navier-Stokes equations coupled to the 2D plate equation with fractional damping accounting for 3D vector displacements. 
Thus, the main results of this paper are three-pronged: 
(1) We provide a hidden regularity result for 2D plates with fractional damping allowing 3D vectoral displacements,
(2) We provide a hidden temporal regularity result for fluid and structure velocities in  a nonlinearly coupled 3D fluid-2D plate FSI problem with fractional damping and 3D vector displacements, and 
(3) We prove a well-posedness result for weak solutions of the 
nonlinearly coupled 3D fluid-2D plate FSI problem with fractional damping and 3D vector displacements.

More precisely,  in terms of spatial regularity of structure displacement, in Section~\ref{sec:spatial_reg} we prove that the structure displacement belongs to the space $L_t^\infty H^{2+\delta}_x$ for a sufficiently small $\delta>0$, {\cblue{which is}} crucial for establishing that the structure displacement is Lipschitz continuous in space at any given time, ensuring injectivity of the maps that map the reference configuration of the fluid domain onto the ``current'' location of the moving domain. 
This {\cblue{is generally one of the}} key {\cblue{issues}} in {\cblue{the analysis of nonlinearly-coupled moving boundary problems with $3D$ (vector) structure displacements. }}
{\cblue{The main ideas behind the proof of this hidden spatial regularity result rely on constructing appropriate test functions for the structure variable and their solenoidal extensions to the fluid domain, which satisfy the kinematic coupling condition. A key step is to formulate a suitable non-homogeneous time-dependent Stokes problem whose solution is used to construct the desirable test functions. 
This approach generalizes the approach presented in \cite{MS22} to vector displacements.
The technique developed here can be applied to other settings, including nonlinear structure operators that are coercive in $H^2$, and different boundary conditions, including the time-dependent inlet/outlet boundary data. 
}}

In terms of temporal hidden regularity result for the fluid and structure velocities, in Section~\ref{subsec:temp} we prove that 
the fractional time derivative of order $1/8$ of the fluid and structure velocities can be uniformly bounded in $L_t^2 L_x^2$,
i.e., we obtain uniform bounds for the fluid and structure velocities in $N_t^{\alpha,2} L_x^2$, where $N^{\alpha,p}$ is Nikolski space.
The key idea is to construct appropriate test functions for the coupled FSI problem by utilizing a time-regularized (averaged) modification of the structure and fluid velocities, similar to the approaches used in \cite{CDEG, G08}.
The construction of these time-regularized (averaged) test functions presents several challenges, arising from the motion of the fluid domain, the non-zero longitudinal displacement of the structure, and the mismatch in spatial regularity between the structure velocity and its corresponding test function. Additionally, the test functions for the fluid and structure must satisfy the kinematic coupling condition at the moving boundary, with the fluid test function also needing to satisfy the divergence-free condition within the moving fluid domain. To enforce these conditions, we construct a Bogovskii-type operator on a time-varying domain with a Lipschitz boundary. The construction of the Bogovskii-type operator presented here
holds significant potential for applications to analyzing general incompressible flow problems on moving domains involving Lipschitz boundaries.

Finally, in Section \ref{sec:exist}, we present a constructive proof of the existence of a local-in-time weak solution to a FSI problem between the 3D flow of an incompressible, viscous fluid modeled by the Navier-Stokes equations and a 2D plate with fractional damping modeling elastodynamics of a plate with 3D vector displacements. We employ a Lie operator splitting method, first utilized in the context of FSI in \cite{MC13} (and further developed in \cite{MC14, MCG20} for different FSI settings). The coupled problem is discretized in time and split into a structure subproblem and a fluid subproblem along the dynamic coupling condition. This time discretization via Lie operator splitting yields a sequence of approximate solutions, which is shown to converge, up to a subsequence and in an appropriate sense, to the desired solution.

The structural regularity result from Section \ref{sec:spatial_reg} is crucial in the construction of approximate solutions and the limiting solution, allowing us to obtain the desired solution up to a strictly positive time $T$ 
determined by self-intersection of the fluid domain boundary. The temporal regularity result obtained in Section \ref{subsec:temp} is crucial to achieve the compactness of the sequence of approximations of the fluid and structure velocities. This ensures that a subsequence of the approximate solutions converges strongly in the relevant topologies as the time step approaches zero, which allowed us to pass to the limit in the approximate weak formulations to prove that the limits satisfy the continuous weak formulation of the original problem. 

In this final step of taking the limit in the approximate weak formations, one needs to deal with one last difficulty associated with general problems on moving domains -- the fact that the test functions in weak formulations depend on the fluid domain motion, and thus on the time-discretization step, via structure displacements,  in a nontrivial way (through the divergence-free condition). This is a classical problem in FSI problems with nonlinear coupling, see e.g., \cite{MC13,MC14,MC16,SunBorSlip}. Taking the limit in approximate weak formulations requires constructing appropriate test functions which would converge, as the time-discretization step converges to zero, to the test functions of the continuous problem in the norm strong enough to pass to the limit. 
Indeed, in Section~\ref{sec:test_functions} we construct such test functions and take the limit in approximate weak formulations to show that the approximate solutions constructed here converge to a weak solution of the continuous problem.


}}

\section{Problem setup}
We  consider the flow a fluid in a periodic channel interacting with a complaint structure that sits atop the fluid domain. See Figure~\ref{fig}.
We assume that the structure displacement is periodic, with the reference domain for the structure equations given by
\begin{equation*}
	\Gamma = \{(x, y, z) \in \mathbb{R}^{3}: (x, y) \in \mathbb{T}^2, z = 1\},
\end{equation*}
where $\mathbb{T}^2$ is the 2D torus.

The fluid reference domain, is then given by 
$$\sO = \Gamma\times (0,1).$$
We denote by $\Gamma_r=\partial\sO\setminus\Gamma$ the rigid part of the boundary of the fluid reference domain $\sO$.
\begin{figure}[h]\centering	\includegraphics[scale=0.5]{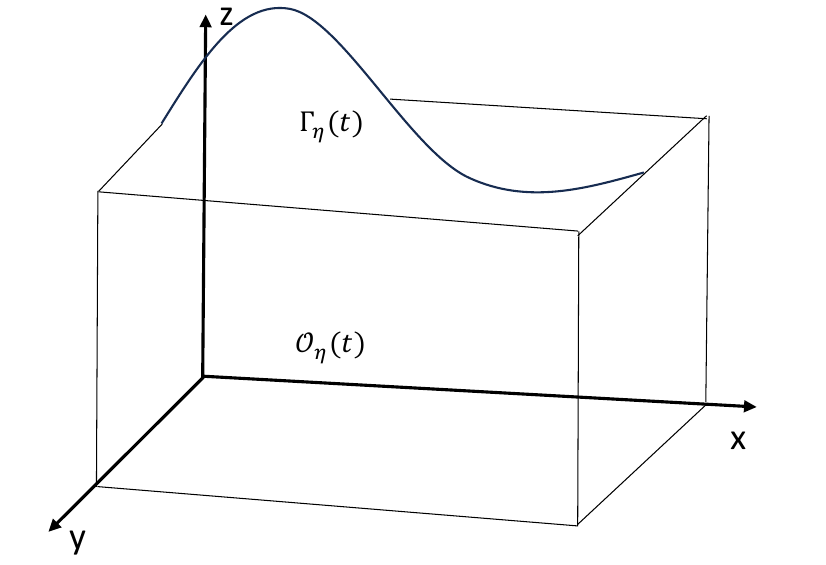}	\caption{The fluid domain}\label{fig}\end{figure}

In this work we assume that the displacement of the compliant structure, denoted by $\bfeta$,  is a vector function with all three components of displacement satisfying a vector equation for a plate with fractional damping, thereby allowing all three components of displacement to be different from zero.

The fluid domain deforms as a result of the interaction between the fluid and the structure. The time-dependent fluid domain in 3D, whose displacement is not known a priori is then given by	
$$\sO_{{\bfeta}}(t)={\bA_\bfeta}(t,\sO),$$ 
whereas its deformable interface is given by 
$$\Gamma_{\bfeta}(t)={\bA_\bfeta}(t,\Gamma),$$
 where $\bA_\bfeta$ is a family of $C^1$ diffeomorphisms  parametrized by time $t\in[0,T]$, such that
\begin{equation}\label{diffeo}
\bA_\bfeta(t)={\bf id}+\bfeta(t)\text{ on } \Gamma,\quad {\bA_\bfeta}(t)|_{\Gamma}={\bf id},\quad \text{det}\nabla{\bA_\bfeta}( t,x,y,z)>0.
\end{equation}

We will now describe the fluid and the structure equations and the two-way coupling that describe the interactions that take place between them.

\vskip 0.1in
\noindent{\bf The fluid subproblem:}		
The fluid flow is modeled by the incompressible Navier-Stokes equations in the 3D time-dependent domains $\sO_{\bfeta}(t) \subset \mathbb{R}^3$ :
\begin{equation}
\left.
	\begin{split}\label{u}
		\partial_t\bu + (\bu\cdot \nabla)\bu &= \nabla \cdot \sigma\\
		\nabla \cdot \bu&=0
	\end{split}
	\right\} \quad {\rm in} \ \sO_{\bfeta}(t) \times (0,T),
\end{equation}
where $\bu=(u_x,u_y,u_z)$ is the fluid velocity. The Cauchy stress tensor is given by $\sigma=-p I+2\nu \bD(\bu)$ where $p$ is the fluid pressure, $\nu>0$ is the kinematic viscosity coefficient and $\bD(\bu)=\frac12(\nabla\bu+(\nabla\bu)^T)$ is the symmetrized gradient of fluid velocity. 
Finally, on the rigid part of the boundary we prescribe  the no-slip boundary conditions:
\begin{align*}
	\bu = 0 \quad\text{ on } \Gamma_r.
\end{align*}
{\bf The structure subproblem:} The elastodynamics problem is given
by the linearly visco-elastic plate equations describing the displacement of the structure in three spatial directions. 
The plate is displaced from its reference domain $\Gamma$ by $\bfeta=(\eta_x,\eta_y,\eta_z)$, which satisfies the following equation for some $0<s\le 1$, 
see e.g., \cite{ChenTriggiani90,tebou2021regularity}:
\begin{align}\label{eta}
	\partial^2_{t}\bfeta +\Delta^2\bfeta +\gamma \Lambda^{2+2s}\partial_t\bfeta = F_\bfeta \quad \text{ in } \Gamma,
	\quad {\rm where} \ \Lambda=(-\Delta)^{\frac{1}{2}}.
\end{align}

Here, $F_{\bfeta}$ denotes the total force experienced by the structure. Assuming that the external forcing on the structure is 0, this force $F_\bfeta$ in the coupled problem results from the jump in the	normal stress (traction) across the structure. With the assumption that the external force is zero,
$F_\bfeta$ comes entirelly from the fluid load felt by the structure (see \eqref{dynbc}).

 Since we work on the torus, the square root of negative Laplatian, denoted here by $\Lambda$, along with its powers can be defined via Fourier transform. 
 
 The damped plate model given by equation \eqref{eta} has been extensively studied in the literature, see e.g.,  \cite{ChenTriggiani90,tebou2021regularity}.
\begin{remark}
	In the classical work \cite{ChenTriggiani90}, $s=0$ is identified as a critical parameter for which the semigroup, {\cblue{defined by the spatial differential operator}}, becomes analytic. In our work, we use the dissipation term to derive a priori estimates, which necessitates that $s>0$.
\end{remark}

We stress here that while the fluid equations are posed on time-dependent domains, in Eulerian framework, the structure equations are defined in Lagrangian coordinates on the fixed reference domain $\Gamma$.

\vskip 0.1in
{\bf The non-linear fluid-structure coupling}: The coupling between the structure and the fluid takes place across the "current" location of the fluid-structure interface. We consider a two-way coupling described by the so-called kinematic and dynamic coupling conditions that describe continuity of velocity and continuity of normal stress at the fluid-structure interface, respectively.
\begin{itemize}
	\item The kinematic coupling condition which describes the continuity of velocities at the interface is the no-slip boundary condition, which in the case of moving boundary, reads:
	\begin{align}\label{kinbc}
		\partial_t \bfeta(t)=(\bu\circ \bA_\bfeta(t))|_\Gamma.
	\end{align}
	\item
	The dynamic coupling condition specifies the load $F_\bfeta$ experienced by the structure:
	\begin{align}\label{dynbc}
		F_{\bfeta}=-S_{\bfeta}(t) \left( (\sigma {\bf n}^{\bfeta}) \circ {\bA_\bfeta}(t)\right)|_\Gamma,
	\end{align}
	where $\bn^{\bfeta}$ is the unit outward normal to the boundary of $\sO_\bfeta$, and {$S_\bfeta=|{\rm cof}\nabla A_{\bfeta}{\bf e}_3|$ defines surface measure of $\Gamma_{{\bfeta}}$, i.e. $d\Gamma_{\bfeta}=S_{\bfeta}d\Gamma$. This term arises from the transformation between Eulerian and Lagrangian coordinates.}
	
\end{itemize}
\subsection{Energy of the coupled problem}\label{energy} In this section we formally derive the following energy inequality corresponding to the coupled fluid-structure interaction problem:
\begin{equation}\label{apriori}
\boxed{	\|\bu\|_{L^\infty(0,T;\bL^2(\sO_\bfeta(\cdot)))} + \|\bD\bu\|_{L^2(0,T;\bL^2(\sO_\bfeta(\cdot)))} + \|\bfeta\|_{L^\infty(0,T;\bH^2(\Gamma))} +\|\partial_t\bfeta\|_{L^2(0,T;\bH^{1+s}(\Gamma))} \leq C}
\end{equation}
\vskip 0.05in
where $C>0$ depends only on the given data $\bu_0,\bv_0,\bfeta_0$.

To show the derivation of this energy inequality, we begin by multiplying the fluid equations \eqref{u} with $\bu$ and then integrate over the moving domain $\sO_\bfeta$. For any $t\in[0,T]$ we obtain,
\begin{align}\label{test}
	\int_{\sO_\bfeta(t)}(\partial_t\bu\cdot\bu + (\bu\cdot\nabla)\bu\cdot\bu) = 	\int_{\sO_\bfeta(t)}(\nabla\cdot\sigma)\cdot\bu.
\end{align}
Thanks to Reynold's transport theorem, the first half of the left-hand side of \eqref{test} can be written as
\begin{align*}
	\int_{\sO_\bfeta(t)}\partial_t\bu\cdot\bu = 	\frac{d}{dt}\int_{\sO_\bfeta(t)}\frac12|\bu|^2-\int_{\Gamma_\bfeta(t)}\frac12|\bu|^2\bu\cdot\bn^\bfeta
\end{align*}
Whereas, the advection term can be treated as follows
\begin{align*}
	\int_{\sO_\bfeta(t)}(\bu\cdot\nabla)\bu\cdot\bu& = \frac12\int_{\Gamma_{{\bfeta}}}|\bu|^2\bu\cdot \bn^{\bfeta}.
\end{align*}
For the term on the right-hand side of \eqref{test} we obtain
\begin{align*}
	\int_{\sO_\bfeta(t)}(\nabla\cdot\sigma)\cdot\bu &= \int_{\partial\sO_\bfeta(t)}\sigma\bn^\bfeta\cdot\bu- \int_{\sO_\bfeta(t)}|\bD\bu|^2.
\end{align*}
Now, by applying the kinematic and dynamic coupling conditions at the fluid-structure interface and by using \eqref{eta} and the fact that $\bfeta$ is periodic in $x$ and $y$, we obtain that
\begin{align*}
	\int_{\Gamma_\bfeta(t)}\sigma\bn^\bfeta\cdot\bu &=\int_\Gamma S_{\bfeta}(t) \left( (\sigma {\bf n}^{\bfeta}) \circ {\bA_\bfeta}(t)|_\Gamma\right) (\bu\circ {\bA_\bfeta}(t))|_\Gamma =\int_\Gamma F_{\bfeta}\partial_t\bfeta\\
	& = \int_\Gamma 	(\partial^2_{t}\bfeta +\Delta^2\bfeta -\gamma\Lambda^{2+2s}\partial_t\bfeta)\cdot\partial_t\bfeta\\
	&= \frac{d}{dt}\int_\Gamma(|\partial_t\bfeta|^2 + |\Delta\bfeta|^2) +\gamma\int_\Gamma|\Lambda^{{1+s}}\partial_t\bfeta|^2.
\end{align*}
Hence, gathering all the equations above we obtain:
\begin{equation}
	\begin{split}
		\frac{d}{dt}\int_{\sO_\bfeta(t)}\frac12|\bu|^2 +\int_{\sO_\bfeta(t)}|\bD\bu|^2+	\frac{d}{dt}\int_\Gamma(|\partial_t\bfeta|^2 &+ |\Delta\bfeta|^2) +\gamma\int_\Gamma|\Lambda^{{1+s}}\partial_t\bfeta|^2.
	\end{split}
\end{equation}
Integration with respect to time implies the energy inequality \eqref{apriori}.

\subsection{Weak formulation on moving domains}
Before we derive the weak formulation of the deterministic system described in the previous subsection, we define the following function spaces
{for the fluid velocity, the structure, and the coupled FSI problem}:
\begin{align*}
	&\tilde\sV_F(t)= \{{\bf u}\in \bH^{1}(\sO_{\bfeta}(t)) :
	\nabla \cdot {\bf u}=0, \bu=0  \text{ on } \Gamma_r	\},\nonumber\\
	&\tilde\sW_F(0,T)=L^\infty(0,T;\bL^2(\sO_\bfeta(\cdot))) \cap L^2(0,T;\tilde\sV_F(\cdot)),\\
	&\sV_S=\bH^2(\Gamma)\\
	&\sW_S(0,T)=W^{1,\infty}(0,T;\bL^2(\Gamma)) \cap L^\infty(0,T;\sV_S) \cap H^1(0,T;\bH^{1+{s}}(\Gamma));\quad 0<s \le 1,\\
	&\tilde\sW(0,T)=\{(\bu,\bfeta)\in \tilde\sW_F(0,T)\times\sW_S(0,T):	\partial_t \bfeta(t)=\bu\circ \bA_\bfeta(t) \text{ on }\Gamma\}.
\end{align*}
Here bold-faced lettered spaces are used for vector valued functions. We will take test functions $(\bq,\bfpsi)$ from the following space:
\begin{equation*}
	\tilde	\sD(0,T)=\{({\bf q},\bfpsi) \in C^1([0,T);\tilde\sV_F(\cdot) \times \sV_S):  \bfpsi(t)=\bq\circ \bA_\bfeta(t), \text{ on }\Gamma\}.
\end{equation*}
Now, we can introduce the weak formulation of our problem on moving domain.
\begin{defn}\label{weakform_moving}
	We say that $(\bu,\bfeta) \in \tilde\sW(0,T)$,  is a weak solution to \eqref{u}-\eqref{dynbc} if for any test function $\bQ=(\bq,\bfpsi) \in \tilde\sD(0,T)$ the following equality holds:
	\begin{equation}
		\begin{split}\label{origweakform}
			&-\int_0^{T }\int_{\sO_{{\bfeta}(t)}}\bu\cdot\partial_t\bq 
			-\int_0^{T}\int_\Gamma\partial_t{\bfeta}\partial_t\bfpsi 
			-\int_0^{T}\int_{\Gamma_{\bfeta}(t)}(\bu\cdot\bq)(\bu\cdot\bn^{\bfeta}) \\
			&+\int_0^T\int_{\sO_{{\bfeta}}(t)}\left( 
			(\bu\cdot\nabla){\bf u}\cdot{\bq}\right)		+ 2\nu\int_0^{T} \int_{\sO_{{\bfeta}(t)}} \bD(\bu)\cdot \bD(\bq) \\
			&	 +\int_0^{T}\int_\Gamma \Delta\bfeta\cdot\Delta \bfpsi  +\gamma \int_0^{T}\int_\Gamma{ \Lambda^{{1+s}}\partial_t\bfeta:\Lambda^{{1+s}} \bfpsi}=\int_{\sO_{\bfeta_0}}\bu_0\bq(0)   + \int_\Gamma \bv_0\bfpsi(0). 
	\end{split}\end{equation}
\end{defn}
\begin{remark}
Notice that under the assumption that there exists of a family of $C^1$ diffeomorphisms $\bA_\bfeta$, defined in \eqref{diffeo}, this weak formulation is well-defined.
Along with the weak solutions, we will construct the corresponding maps $\bA_\bfeta$ satisfying these assumptions. 
\end{remark}
We recall that one of our main results in this manuscript is the existence of a solution to the FSI problem \eqref{u}-\eqref{dynbc} in the sense of Definition \ref{weakform_moving}.

\subsection{Arbitrary Lagrangian-Eulerian (ALE) formulation on fixed domain}\label{subsec:ALE}
To deal with the geometric non-linearities resulting from the motion of the fluid domain we transform the fluid equations onto the fixed reference domain $\sO=\Gamma\times (0,1)$ and give a weak formulation equivalent to \eqref{origweakform} posed on this fixed domain. For that purpose, we consider 
a family of Arbitrary Lagrangian-Eulerian (ALE) mappings that are ubiquitous in the field of computational fluid-structure interaction. The ALE maps, denoted by $\bA_\bfeta$, constitute a family, parametrized by time $t\in[0,T]$, of diffeomorphisms from the fixed domain $\sO$ onto the moving domain $\sO_{{\bfeta}}(t)$. With the aid of these maps, we will find a relevant weak formulation on $\sO$ satisfied by $\bu\circ \bA_\bfeta$. 

 In this article, these maps will be obtained by considering harmonic extensions of the structure displacement $\bfeta$ in $\sO$. That is, the ALE maps solve the following equations:
\begin{equation}\label{ale}
	\begin{split}
		\Delta \bA_\bfeta&=0, \quad \text{ in } \sO,\\
		\bA_\bfeta=\textbf{id}+{\bfeta} \text{ on } \Gamma,& \quad \text{ and }\quad
		\bA_\bfeta=\textbf{id}  \text{ on } \partial\sO\setminus\Gamma.
	\end{split}
\end{equation}
The existence and uniqueness of the solution to \eqref{ale} is classical. However, we need to prove that these maps are well-defined, i.e. $\bA_\bfeta(t): \sO\to\sO_\bfeta(t)$ is indeed a $C^1$-diffeomorphism for every $t\in[0,T]$ (see Remark \ref{equiv}).

Before moving on to analyzing the properties of $\bA_\bfeta$, we summarize the notation that will be used to simplify the ALE formulation of the problem. First, we denote the Jacobian of the ALE maps by
\begin{equation}\label{ALE_Jacobian}
	J_{\bfeta}=	\text{det }\nabla \bA_{\eta}.
\end{equation}
Next, under the transformation given in \eqref{ale}, the transformed gradient and the transformed symmetrized gradient of any function $\bg^\bfeta:=\bg\circ \bA_\bfeta$ for $\bg\in \bH^1(\sO_\bfeta)$ are given by
$$\nabla^{\bfeta} \bg^{\bfeta}=\nabla \bg\circ \bA_\bfeta=\nabla\bg^{\bfeta}(\nabla\bA_\bfeta)^{-1} \quad \text{ and }\quad \bD^\eta(\bu)=\frac12(\nabla^\eta\bu+(\nabla^\eta)^T\bu). $$
Similarly, the transformed divergence will be denoted by
$$ \text{div}^{\bfeta}\bu=\nabla^{\bfeta}\cdot\bu=tr(\nabla^\bfeta\bu).$$
Finally, we use $\bw^\eta$ to denote the ALE velocity:
$$\bw^\eta=\partial_t\bA_\bfeta .$$
Using this notation we will give the definition of function spaces used to describe the fixed domain ALE formulation of our FSI problem.
We define,
\begin{align*}
	&\sV^{\bfeta}_F(t)= \{{\bf u}\in \bH^{1}(\sO) :
	\nabla^\bfeta \cdot {\bf u}=0, \bu=0  \text{ on } \Gamma_r	\},\nonumber\\
	&\sW_F(0,T)=L^\infty(0,T;\bL^2(\sO)) \cap L^2(0,T;\sV^{\bfeta}_F(\cdot)),\\
	&\sW(0,T)=\{(\bu,\bfeta)\in \sW_F(0,T)\times\sW_S(0,T):	\partial_t \bfeta(t)=\bu|_{\Gamma}\}.
\end{align*} 

The space of test functions is as follows:
\begin{equation*}
	\sD^\bfeta(0,T)=\{({\bq},\bfpsi) \in C^1([0,T);\sV^{\bfeta}_F(\cdot)\cap H^3(\sO) \times \sV_S) : {\bf q}|_{\Gamma}=\bfpsi\}.
\end{equation*}
Now we will now present a weak formulation on the fixed domain $\sO$, derivation of which is the same as given in Section 4.3 \cite{MC13}.
\begin{defn}\label{weakform_steady}
	We say that $(\bu,\bfeta)\in \sW(0,T)$ is a  weak solution of the nonlinearly coupled FSI problem \eqref{u}-\eqref{dynbc} defined in terms of a {\bf{fixed domain formulation}}
	on $\sO$ if the following equation holds for any $(\bq,\bfpsi)\in \sD^\bfeta(0,T)$:
	\begin{equation}
		\begin{split}\label{weaksol}
			&-\int_0^T\int_\sO{ J}_{\bfeta}\bu\cdot \partial_t{\bq} -\int_0^T\int_\Gamma \partial_t {\bfeta} \,\partial_t{\bfpsi}= \int_0^T\int_{\sO}\partial_tJ_\bfeta 
			\bu\cdot\bq \\
			& -\int_0^T\int_{\sO}J_{\bfeta}(\bu\cdot\nabla^{\bfeta }\bu \cdot\bq
			- \bw^{\bfeta}\cdot\nabla^{\bfeta }\bu\cdot\bq )
			-2\nu\int_0^T\int_{\sO}{J}_{\bfeta}\, \bD^{{\bfeta} }(\bu ): \bD^{{\bfeta} }(\bq)\\
			&+\int_0^{T}\int_\Gamma \Delta\bfeta\cdot\Delta \bfpsi+\gamma  \int_0^{T}\int_\Gamma \Lambda^{{1+s}}\partial_t\bfeta:\Lambda^{{1+s}} \bfpsi+\int_{\sO}J_{\bfeta_0}\bu_0\bq(0)   + \int_\Gamma \bv_0\bfpsi(0).
		\end{split}
	\end{equation}
\end{defn}
\begin{remark}\label{equiv}
	We note that if the ALE map $\bA_\bfeta(t): \sO \mapsto \sO_\bfeta(t)$, defined as the solution to \eqref{ale}, {is Lipschitz continuous} and bijective, then Definitions \ref{weakform_moving} and \ref{weakform_steady} are equivalent. In other words, $(\bu,\bfeta)$ solves \eqref{weaksol} iff $(\tilde\bu,\bfeta)$ where $\tilde\bu= \bu\circ \bA^{-1}_\bfeta$ solves \eqref{origweakform} i.e. it is the desired weak solution of our FSI problem in the sense of Definition \ref{weakform_moving}. 
\end{remark}

\begin{remark}\label{notation}{\bf{(Notation)}}
Throughout the rest of the manuscript we will be using  $\tilde\bu$, where $$\tilde\bu= \bu\circ \bA^{-1}_\bfeta,$$ 
to denote the fluid velocity defined on the moving domain $\sO_\bfeta$, to distinguish between the solution $\bu$ defined on the fixed domain $\sO$, 
and the solution $\tilde\bu$ defined on the moving domain $\sO_\bfeta$. 
\end{remark}

	Next, we discuss conditions that are sufficient to imply bijectivity of the ALE maps $\bA_\bfeta$.
	First, observe that for any $k \geq 0$ the solution to \eqref{ale} satisfies (see e.g. \cite{G11}):
	\begin{align}\label{ale_Hs}	\|\bA_\bfeta\|_{\bH^{k+\frac12}(\sO)} \leq \|\bfeta\|_{\bH^k(\Gamma)}.
	\end{align}
	Observe also that, for any $p\geq2$, we have the following regularity result for the harmonic extension $\bA_\bfeta$ of the boundary data ${\bf id} +\bfeta$ thanks to the discussion presented in Section 5 in \cite{G11}:
	\begin{align}\label{boundsA1}
		\|\bA_\bfeta-{\bf id}\|_{\bW^{2,p}(\sO)} \leq C\|{\bfeta}\|_{\bW^{2-\frac1{p},p}(\Gamma)} \leq C\|{\bfeta}\|_{\bH^{3-\frac3p}(\Gamma)} .
	\end{align}
	Hence, Morrey's inequality (see e.g. Theorem 7.26 in \cite{GT83}) implies that for some $C^*_p>0$ the following inequality holds true for $p>3$,
	\begin{align}\label{boundJ}
		\|\nabla (\bA_\bfeta-{\bf id})\|_{\bC^{0,1-\frac3p}( \bar\sO)}\leq 	\|\nabla (\bA_\bfeta-{\bf id})\|_{\bW^{1,p}( \bar\sO)} \leq C^*_p\|{\bfeta}\|_{\bH^{3-\frac3p}(\Gamma)}.
	\end{align}
	Now thanks to Theorem 5.5-1 (B) of \cite{C88}, for as long as the structure displacement $\bfeta$ satisfies 
	\begin{align}\label{bound_injective}
		\|\bfeta\|_{\bH^{3-\frac3p}(\Gamma)} \leq \frac1{C^*_p} , \qquad \text{for any } p>3,
	\end{align}
	the map $\bA_\bfeta\in \bC^{1,1-\frac3p}(\bar\sO)$ is injective. Thanks to invariance of domains (see \cite{C88}), we infer that $\bA_\bfeta$ is thus a bijection between the domains $\sO$ and $\sO_\bfeta$.

We have established the following
\begin{proposition}\label{equiv}
If for any $\delta>0$, $\bfeta\in L^\infty(0,T;\bH^{2+\delta}(\Gamma))$ then for some small enough $T_0>0$, the map $\bA_\bfeta\in L^\infty(0,T_0;\bC^{1,\delta}(\bar\sO))$ solving \eqref{ale} is bijective and Definitions \ref{weakform_moving} and \ref{weakform_steady} are thus equivalent.
\end{proposition}

\subsection{Main results}
We are now in a position to state the main results of this article. 

In the first two theorems we will state the 
enhanced spatial regularity of the structure and the enhanced temporal regularity of the fluid and structure velocities. 
 
Before stating these results we recall the definition of Nikolski spaces. 	Let	the translation in time by $h$ of a function $f$ be denoted by:
$$\tau_h f(t,\cdot)=f(t-h,\cdot), \quad h\in \R.$$
Let $\mathcal{Y}$ be a Banach space.
Then, for any $0<m<1$ and $1\leq p<\infty$, the Nikolski space is defined as:
\begin{equation}\label{Nikolski}
	{N^{m,p}(0,T;\mathcal{Y})}=\{\bu\in L^p(0,T;\mathcal{Y}):\sup_{0<h< T} \frac1{h^m} \|\tau_h\bu-\bu\|_{L^p(h,T;\mathcal{Y})}<\infty\}.
\end{equation}
Now we state our a priori estimates that provide additional regularity  for the structure displacement and the fluid velocity. These are our first two main results of the manuscript. 
\begin{thm}\label{thm:Lip1}
	Let $(\tilde\bu,\bfeta)$ be a smooth solution to the FSI problem defined on moving domains, satisfying \eqref{origweakform}.
	Then the following a priori estimates addressing spatial regularity hold true:
	\begin{enumerate}
		\item The structure displacement $\bfeta$ satisfies:
		\begin{align}\label{reg_eta}
			\|\bfeta\|_{L^\infty(0,T_0;\bH^{2+\delta}(\Gamma))} +\|\bfeta\|_{L^2(0,T_0;\bH^{3-(s-\delta)}(\Gamma))} <C,\qquad \text{for any }\, 0<\delta< s.
		\end{align}
		\item Moreover, the ALE maps defined by \eqref{ale} satisfy
		\begin{align}\label{reg_ale}
			\|\bA_\bfeta\|_{L^\infty(0,T_0;\bC^{1,\delta}(\bar\sO))} <C.	
		\end{align}
		\end{enumerate}
		Here, $C$ depends only on the energy norm of the initial data, as well as the $H^{2+s}$-norm of the initial displacement $\bfeta_0$, 
		and on the viscoelasticity coefficient $\gamma > 0$ and on domain $\sO$. 
\end{thm}
		
\begin{thm}\label{thm:Lip2}		
		Let $(\tilde\bu,\bfeta)$ be a smooth solution to the FSI problem defined on moving domains, satisfying \eqref{origweakform}.
	Then
		the  fluid and structure velocities $(\tilde\bu,\partial_t\bfeta)$ 
				 satisfy the following a priori estimate addressing temporal regularity property:
		\begin{align}\label{reg_temp}
			\|\tilde\bu\circ\bA_\bfeta\|_ {N^{\frac18,2}(0,T_0;\bL^2(\sO))} + \|\partial_t\bfeta\|_{N^{\frac18,2}(0,T_0;\bL^2(\Gamma))}< C.
		\end{align}
		Here, $C$ depends only on the energy norm of the initial data, as well as the $H^{2+s}$-norm of the initial displacement $\bfeta_0$, 
		and on domain $\sO$. 
\end{thm}

Our third main result of the manuscript is the existence of a weak solution to the nonlinearly coupled problem, as stated in the following theorem.

\begin{thm}\label{thm:exist}
	Let the initial data for structure displacement, structure velocity and fluid velocity be such that  $\bfeta_0 \in \bH^{2+s}(\Gamma),\bv_0\in\bL^2(\Gamma)$ and $\bu_0\in \bL^2(\sO_{\bfeta_0})$. Then there exists $T_0>0$ 
	and at least one weak solution to the system \eqref{u}-\eqref{dynbc} on $[0,T_0]$ in the sense of Definition \ref{weakform_moving}.
\end{thm}

In what follows, we will give the proofs of these two theorems. We will start, in Section \ref{sec:reg}, with the proofs Theorems \ref{thm:Lip1} and \ref{thm:Lip2},
and then use these regularity results in Section \ref{sec:exist} to construct a weak solution for \eqref{u}-\eqref{dynbc}, thus proving Theorem \ref{thm:exist}. Specifically, Theorem \ref{thm:Lip1}, which states that at any time the structure displacement is Lipschitz continuous in space, is crucial in obtaining a positive time-length during which the fluid domain remains non-degenerate and thus in transforming the fluid equations onto the fixed domain $\sO$. It is also used in the construction of the Bogovski-type operator constructed in the proof of Theorem \ref{thm:Lip2}. Theorem \ref{thm:Lip2} is used in Section \ref{sec:exist} to obtain compactness of the sequence of approximate solutions to prove the existence of a solution to the FSI problem in the sense of Definition \ref{weakform_steady}.
\section{Regularity results}\label{sec:reg}
\subsection{The structure regularity result.}\label{sec:spatial_reg}

In this section we will prove Theorem \ref{thm:Lip1} showing the a priori regularity result for $\bfeta$. To establish this result we work in the fixed domain setting of Definition \ref{weakform_steady} and operate under the assumption that $\bfeta$ is smooth and that the map $\bA_\bfeta(t):\sO\mapsto\sO_\bfeta(t)$, solving \eqref{ale} is bijective.

In this case, we make note of the following result that gives us the equivalent of the energy estimate \eqref{apriori} for the fixed domain counterparts.
\begin{lem}\label{uest} 
	Let $(\bu,\bfeta)$ be a weak solution in the sense of the Definition \ref{weakform_steady} on the fixed domain $\sO$. Assume that the ALE maps $\bA_\bfeta(t):\sO\mapsto\sO_\bfeta(t)$, solving \eqref{ale}, are bijective and that for some $\alpha>0$ their Jacobians satisfy $\inf_\sO J_\bfeta>\alpha>0$ for all $t\in[0,T]$.
	Then, for some constant $K_1>0$ depending only on $\|\bA_\bfeta\|_{L^\infty(0,T;\bW^{1,\infty}(\sO))}, \|(\bA_\bfeta)^{-1}\|_{L^\infty(0,T;\bW^{1,\infty}(\sO))}$ and $\alpha$ we have,
	$$\|\bu\|_{L^\infty(0,T;\bL^2(\sO))\cap L^2(0,T;\sV_F^{\bfeta})} 
	<K_1.$$ 
\end{lem}
\begin{proof}[Proof of Lemma \ref{uest}]
	This Lemma is a consequence of the energy estimate \eqref{apriori}. Owing to the assumption that $\bA_\bfeta(t):\sO\mapsto\sO_\bfeta(t)$ is bijective, we can write
	$$\bu = \tilde\bu\circ \bA_\bfeta,$$
	where $(\tilde\bu,\bfeta)$ is a solution to the FSI problem in the sense of Definition \ref{weakform_moving}.
	Then the energy estimate \eqref{apriori} gives us that 
	$$\sup_{t\in[0,T]}\int_{\sO_\bfeta}|\tilde\bu|^2+\int_0^T\int_{\sO_\bfeta}|\bD(\tilde\bu)|^2 \leq C,$$
	where $C$ depends only on the given initial data.

We will use	these bounds to obtain the desired estimates for $\bu$. The first bounds are obtained easily by a change of variables as follows,
\begin{align*}
\alpha \sup_{0\le t\leq T} \int_{\sO}|\bu|^2\leq 	\sup_{0\le t\leq T} \int_{\sO}J_\bfeta|\bu|^2 =	\sup_{0\le t\leq T} \int_{\sO_\bfeta}|\tilde\bu|^2 \leq C. 
\end{align*}

 Next, to bound $\bu$ in $L^2(0,T;\sV_F^{\bfeta})$,
 we must first establish a connection between the gradient and the symmetrized gradient of $\tilde{\bu}$ which is traditionally done with the aid of Korn's inequality. However, due to our setting that involves time-varying fluid domains, we appeal to Lemma 1 in \cite{V12} that gives the existence of a universal Korn constant $K>0$ which depends only on the reference domain $\sO$ and the quantities $\|\bA_\bfeta\|_{L^\infty(0,T;\bW^{1,\infty}(\sO))}, \|(\bA_\bfeta)^{-1}\|_{L^\infty(0,T;\bW^{1,\infty}(\sO))}$. For this constant we have that
	\begin{align*}
		\|\tilde\bu\|_{\bH^1(\sO_\bfeta)} \leq K\|\bD(\tilde\bu)\|_{\bL^2(\sO_\bfeta)}.
	\end{align*}
	These bounds do not immediately translate to desired $L^2_tH^1_x$-bounds for $\bu$. We observe that on the fixed $\sO$ we have the following relation between the gradient and the transformed gradient (via ALE maps) of $\bu$:
	$$\nabla\bu=\nabla^{\bfeta}\bu \cdot\nabla \bA_\bfeta.$$
	Hence, we write, 
	\begin{align*}
		\alpha \int_0^T&\int_{\sO}|\nabla\bu|^2dx \leq  \int_0^T\int_{\sO}J_\bfeta|\nabla\bu|^2dx
		=\int_0^T\int_{\sO}J_\bfeta|\nabla^{\bfeta}\bu\cdot \nabla \bA_\bfeta|^2dx\\
		&\leq \|\bA_\bfeta\|^2_{L^\infty(0,T;\bW^{1,\infty}(\sO))}\int_0^T\int_{\sO}J_\bfeta|\nabla^{\bfeta}\bu|^2dx=\|\bA_\bfeta\|^2_{L^\infty(0,T;\bW^{1,\infty}(\sO))}\int_0^T\int_{\sO_\bfeta}|\nabla^{}\tilde\bu|^2dx\\
		&\leq K\|\bA_\bfeta\|^2_{L^\infty(0,T;\bW^{1,\infty}(\sO))}\int_0^T\int_{\sO_\bfeta}|\bD(\tilde\bu)|^2dx\\
		& \leq K_1(\|\bA_\bfeta\|_{L^\infty(0,T;\bW^{1,\infty}(\sO))}, \|(\bA_\bfeta)^{-1}\|_{L^\infty(0,T;\bW^{1,\infty}(\sO))}).
	\end{align*}
	This completes the proof of Lemma \ref{uest}.
\end{proof}
Now, we proceed with the proof of Theorem \ref{thm:Lip1}. We will assume that the setting of Lemma \ref{uest} holds true. 
The main idea behind the proof of Theorem \ref{thm:Lip1}, namely obtaining estimate \eqref{reg_eta},  is to consider the "transformed" weak formulation \eqref{weakform_steady}
and use for a test function $\bfpsi$ the function $\bfpsi \sim \Lambda^{2\kappa}\bfeta$ for $1-s<\kappa<1$. In fact, to obtain precisely \eqref{reg_eta},  we take 
\begin{equation}\label{kappa}
\kappa = 1-(s-\delta),\qquad\text{for any $0<\delta<s$}.
\end{equation}
Due to the kinematic coupling condition embedded in the test space $\sD^\bfeta(0,T)$, we will also construct a transformed-divergence (div$^\bfeta$)-free extension $\bq$ of $\bfpsi$ to be used as the fluid test function.  In the setting of \cite{MS22}, where tangential interactions between the fluid and the structure are negligible, this extension, in the case of flat reference geometry, is obtained simply by extending the boundary data onto the moving domain by a constant in the direction normal to $\Gamma$ and then composing it with the ALE map.
However, constructing an appropriate extension in our setting is not easy. Firstly, $\Lambda^{2\kappa}\bfeta$ is not guaranteed to have a solenoidal extension in the moving fluid domain $\sO_\bfeta$ and thus special care has to be taken in the construction of $\bfpsi$ to ensure that it possesses an extension $\bq$ in $\sO$ 
which is divergence free in terms of the transformed-divergence operator (div$^\bfeta$). Secondly, $\bq$ and $\bfpsi$, as a pair of test functions for \eqref{weakform_steady}, must satisfy appropriate bounds.

Now, due to its complicated form, instead of looking for a transformed-divergence-free extension $\bq$ of the function $\bfpsi$ directly, we will first find a solenoidal extension of a modification of $\bfpsi$, denoted by $\bfvarphi$, on $\sO$, and then transform this function appropriately to obtain the desired  $\bq$. That is, we will find a function $\bfvarphi$ such that it satisfies div$\bfvarphi=0$ on $\sO$ and then we will transform $\bfvarphi$ into $\bq$ in a way that guarantees that div$^\bfeta\bq=0$. This transformation will be obtained by multiplying $\bfvarphi$ with the inverse of the cofactor matrix of $\bA_\bfeta$ and using the Piola identity (see Theorem 1.7-1 in \cite{C88}) to obtain (see \eqref{divq}):
\begin{align}\label{transformdiv}
	\bq &= J^{-1}_{\bfeta}(\nabla \bA_\bfeta) \bfvarphi.
\end{align}
At this point we only have $\bq$ written in terms of $\bfvarphi$, but we still do not have $\bfvarphi$ defined, and we still do not have $\bfpsi$.
Next, we work on constructing the test function $\bfpsi$ that "behaves" like $\Lambda^{2\kappa}\bfeta$ and satisfies the kinematic coupling condition with $\bq$,
and has the additional property that  its appropriate modification has a divergence-free extension $\bfvarphi$ in $\sO$. 

Naturally, this modification must account for the transformation of $\bfvarphi$ into $\bq$ as given in \eqref{transformdiv}. Hence, we define
\begin{align}
	\bfpsi: = \Lambda^{2\kappa}\bfeta -c\bfxi
\end{align}
where $c\bfxi$ is a correction term that allows us to transform $\bfpsi$ so that its transformation possesses a divergence free extension in $\sO$. More precisely, we let
\begin{align}\label{defc}
	{c}
	= \frac{\int_{\Gamma}\nabla\bfeta \times \Lambda^{2\kappa}\bfeta}{\int_{\Gamma}\nabla\bfeta\times\bfxi},
\end{align}
where 
$\bfxi \in \bC^\infty_0([0,T]\times\Gamma)$ is such that the denominator in the definition of the constant $c$ is non-zero. In fact, we  choose $\bfxi$ such that $\|\bfxi(t)\|_{\bC^2(\Gamma)}=1$ and $\nabla\bfeta\times\bfxi(t) =1$ for every $t\in[0,T]$. Note that for this choice of $\bfxi$ we have
\begin{align*}
	\sup_{0\le t\leq T} |c(t)| \leq 	\sup_{0\le t\leq T} \|\bn_\bfeta(t)\|_{\bL^2(\Gamma)}\|\Lambda^{2\kappa}\bfeta(t)\|_{\bL^2(\Gamma)}.
\end{align*}
Note, due to the periodic boundary conditions imposed on the structure displacement, $\bfpsi$ is indeed a valid structure test function.

Now for the solenoidal function $\bfvarphi$ in \eqref{transformdiv}, we define it to be the solution of a time-dependent Stokes problem with non-homogeneous boundary data defined as follows. 

For any fixed $\delta$ such that $0<\delta<s$,  let $\kappa$ be as defined in \eqref{kappa}, namely $\kappa = 1-(s-\delta)$.
Then we choose the solenoidal function $\bfvarphi$ in \eqref{transformdiv} to be the solution of
\begin{equation}\begin{split}\label{stokes}
		\bfvarphi_t	-\Delta \bfvarphi +\nabla p&= 0 \qquad \text{ in } \sO,\\
		\text{div }\bfvarphi &= 0 \qquad \text{ in } \sO,\\
		\bfvarphi & = J_{\bfeta} (\nabla \bA_\bfeta)^{-1}|_{\Gamma}\left( \Lambda^{2\kappa}\bfeta - c\bfxi\right) \qquad \text{ on } \Gamma,	\\
		\bfvarphi(t=0)&=\bfvarphi_0,
\end{split}	\end{equation}
such that initial condition satisfies,
\begin{align*}
	\bfvarphi_0|_\Gamma &= J_{\bfeta_0} (\nabla A_{\bfeta_0})^{-1}|_{\Gamma}\left( \Lambda^{2\kappa}\bfeta_0 - c\bfxi(0)\right) .
\end{align*}

Indeed, observe that this is the right choice of boundary value since it satisfies the following compatibility condition:
$$\int_{\Gamma}\bfvarphi \cdot (0,0,1) = \int_{\Gamma} 
\nabla\bfeta\times (\Lambda^{2\kappa}\bfeta-c\bfxi)=0.$$
Now, for an appropriate {{choice of the}} trace space $\mathcal{G}^{m}(\Gamma\times(0,T))$,  Theorem 6.1 in \cite{FGH02} guarantees the existence of a unique solution $(\bfvarphi, p)$ to \eqref{stokes} that satisfies
\begin{align}
	\|\bfvarphi\|_{L^2(0,T;\bH^m(\sO))}+\|\partial_t \bfvarphi\|_{L^2(0,T;\bH^{m-2}(\sO))}&+\|\nabla p\|_{L^2(0,T;\bH^{m-2}(\sO))} \notag\\
	&\leq \|J_{\bfeta}(\nabla \bA_\bfeta)^{-1}|_{\Gamma}(\Lambda^{2\kappa}\bfeta-c\bfxi)\|_{\mathcal{G}^{m}(\Gamma\times(0,T))}.\label{reg_stokes}
\end{align}
We will consider $\mathcal{G}^{m}(\Gamma_T)$ with $\frac32 < m < 2$. This choice of $m$ balances the following two considerations: 
the chosen $m$ has to be large enough to bound the time derivative of the test function $\bq$ in an appropriate dual space, which will be discussed later in estimate \eqref{I5} (see the remark following the estimate), while still ensuring that $m$ is not too large in order to capture the limited regularity of the boundary data in \eqref{stokes}.

For any $m > \frac32$ the trace space $\mathcal{G}_m$ is endowed with the following norm,
\begin{align*}
	\|\phi\|_{\mathcal{G}^{m}(\Gamma\times(0,T))} := \|\phi\|_{L^2(0,T;\bH^{m-\frac12}(\Gamma))}+ \|\phi\cdot\bn\|_{ H^1(0,T;\bH^{m-\frac52}(\Gamma))}+\|\phi_{\boldsymbol{\tau}}\|_{H^{\frac{2m-1}{2m}}(0,T;\bH^{(1-\frac2m)(m-\frac12)}(\Gamma))} ,
\end{align*}
where $\bn = (0,0,1)$ is the unit normal to $\Gamma$ and $\phi_{\boldsymbol{\tau}}$ is the projection of $\phi$ onto the tangent space of $\Gamma$. 

Next we comment on the validity of the choice of such test functions $(\bq,\bfpsi)$.
In summary, we have defined
\begin{align}\label{testfunction}
	\bq = -J_{\bfeta}^{-1}\nabla \bA_\bfeta\,\bfvarphi,\qquad \bfpsi = -(\Lambda^{2\kappa}\bfeta - c\bfxi),
\end{align}
where $c$ is given in \eqref{defc} and $\bfvarphi$ solves \eqref{stokes}. As mentioned earlier, due to the properties of the Piola transform (see e.g. Theorem 1.7-1 in \cite{C88}), we have
\begin{align}\label{divq}
	\text{div}^{\bfeta}\bq = J_{\bfeta}^{-1}(\text{div}\bfvarphi)=0.
\end{align}
Moreover, it is also true that $\bq|_\Gamma=\bfpsi$ on $(0,T)\times\Gamma$. Hence, we conclude that this pair $(\bq,\bfpsi)$ is a valid test function for \eqref{weaksol}.

We proceed with the proof of Theorem~\ref{thm:Lip1} by replacing  the test function $\bfpsi$
in the weak formulation \eqref{weaksol} with the above-constructed $\bfpsi=-(\Lambda^{2\kappa}\bfeta - c\bfxi)$,
and then express the terms containing $\bfeta$ that we want to estimate, using the remaining terms from the weak formulation. We obtain:
\begin{equation}\label{testing}
	\begin{split}
		&\frac{\gamma}2\int_0^{T}\int_\Gamma \frac{d}{dt}|\Lambda^{1+\kappa+s}\bfeta|^2+
		\int_0^{T}\int_\Gamma |\Lambda^{2+{\kappa}}\bfeta|^2= \int_0^T c\int_\Gamma |\partial_t\Lambda^{\kappa} {\bfeta}|^2 \\
		&+ \int_0^T\int_\Gamma \partial_t {\bfeta} \,\partial_t{ \bfxi} +\int_0^{T}c\int_\Gamma \Delta\bfeta\cdot\Delta \bfxi-\gamma\int_0^{T }c\int_\Gamma \Lambda^{1+s}\partial_t\bfeta:\Lambda^{1+s}\bfxi \\
		&+\int_0^T\int_\sO{ J}_{\bfeta}\bu\cdot \partial_t{\bq} -\int_0^T\int_{\sO}J_{\bfeta}(\bu\cdot\nabla^{\bfeta }
		\bu \cdot\bq
		- \bw\cdot\nabla^{\bfeta }\bu\cdot\bq )\\
		&+\int_0^T\int_{\sO}
		\partial_tJ_{\bfeta}\bu\cdot\bq 
		-2\nu\int_0^T\int_{\sO}{J}_{\bfeta}\, \bD^{{\bfeta} }(\bu )\cdot \bD^{{\bfeta} }(\bq) \\
		&+\int_{\sO}J_{\bfeta_0}\bu_0\bq(0)   + \int_\Gamma \bv_0\bfpsi(0) ,\\
		&:= I_1+...+I_{11}.
	\end{split}
\end{equation}
In the rest of this proof we will estimate the terms $I_j, 1\leq j\leq 11$ to get the desired final estimate.

However, before estimating each term $I_j, 1\leq j\leq 11$,
we plan to obtain bounds for $\bfvarphi$ that will result in appropriate bounds for the test function $\bq$ (see \eqref{testfunction}),
which will require bounds for $\bfvarphi|_\Gamma=J_{\bfeta} (\nabla \bA_\bfeta)^{-1}|_{\Gamma}\left( \Lambda^{2\kappa}\bfeta - c\bfxi\right)$ in the trace space 
$\mathcal{G}^m(\Gamma\times(0,T))$ for a well-chosen $m>\frac32$.
More precisely, we plan to use \eqref{reg_stokes} to show the following estimate of $\bfvarphi|_\Gamma$:


\begin{proposition}\label{phi_trace_est}
For any $0<\delta < s$ let $m=\frac32+\ep$ where $0<\ep< \min\{\delta/4, s-\delta\}$.       
Then, the function $\bfvarphi$ defined to be the solution of \eqref{stokes}, satisfies the following trace estimate:
\begin{align}\label{varphitrace}
	\|\bfvarphi|_\Gamma\|_{\mathcal{G}^m(\Gamma\times(0,T))  }\leq 	C(1+ \|\bfeta\|^{\frac12}_{L^\infty(0,T;\bH^{2+\delta}(\Gamma))}\|\bfeta\|_{L^2(0,T;\bH^{2+\kappa}(\Gamma))}), 
\end{align}
where $\kappa$ is defined in \eqref{kappa}.
\end{proposition}
\begin{proof}(Proof of Proposition~\ref{phi_trace_est})
First, we write $m=\frac32+\ep$.
Now, observe that for any $0<\ep<1-\kappa=s-\delta$, the following estimate holds true:
\begin{align*}
	\|\Lambda^{2\kappa}\bfeta\|_{L^2(0,T;\bH^{m-\frac12}(\Gamma))}=\|\Lambda^{2\kappa}\bfeta\|_{L^2(0,T;\bH^{1+\ep}(\Gamma))} = \|\bfeta\|_{L^2(0,T;\bH^{2\kappa+1+\ep}(\Gamma))} \leq \|\bfeta\|_{L^2(0,T;\bH^{2+\kappa}(\Gamma))}.
\end{align*}
Due to the trace theorem, the fact that $H^{\frac32+\ep}(\sO)$ is a Banach algebra, and using the Sobolev estimate for harmonic extensions \eqref{ale_Hs}, we have
\begin{align*}
	\|J_\bfeta(\nabla \bA_\bfeta)^{-1}|_\Gamma\|_{\bH^{1+\ep}(\Gamma)}\leq C\|J_\bfeta(\nabla \bA_\bfeta)^{-1}\|_{\bH^{\frac32+\ep}(\sO)} \leq C\|\bA_\bfeta\|^2_{\bH^{\frac52+\ep}(\sO)}\leq C\|\bfeta\|^2_{\bH^{2+\ep}(\Gamma)}.
\end{align*}
We now interpolate the right-hand side as follows,
\begin{align}
	\|J_\bfeta(\nabla \bA_\bfeta)^{-1}|_\Gamma\|_{L^\infty(0,T;\bH^{1+\ep}(\Gamma))} &	\leq C\|\bfeta\|^2_{L^\infty(0,T;\bH^{2+\ep}(\Gamma))}\notag\\
	&\leq C\|\bfeta\|^{\frac32}_{L^\infty(0,T;\bH^{2}(\Gamma))}\|\bfeta\|^{\frac12}_{L^\infty(0,T;\bH^{2+4\ep}(\Gamma))}\notag,\\
	&\leq C\|\bfeta\|^{\frac12}_{L^\infty(0,T;\bH^{2+4\ep}(\Gamma))}. \label{JA}
\end{align}
Hence, choosing 
\begin{equation}\label{epsilon}
0<\ep < \min\{\frac{\delta}4,1-\kappa\}
\end{equation} 
and noting that for $s>1$, $H^s(\Gamma)$ is a Banach algebra, we obtain
\begin{align}
	\notag\|\bfvarphi|_\Gamma\|_{L^2(0,T;\bH^{m-\frac12}(\Gamma))} &\leq C
	\|J_\bfeta(\nabla \bA_\bfeta)^{-1}\Lambda^{2\kappa}\bfeta|_\Gamma\|_{L^2(0,T;\bH^{1+\ep}(\Gamma))}\\
	\notag	 &\leq C\|J_\bfeta(\nabla \bA_\bfeta)^{-1}|_\Gamma\|_{L^\infty(0,T;\bH^{1+\ep}(\Gamma))}\|\Lambda^{2\kappa}\bfeta\|_{L^2(0,T;\bH^{1+\ep}(\Gamma))}\\
	\notag	&\leq C\|\bfeta\|^{\frac12}_{L^\infty(0,T;\bH^{2+4\ep}(\Gamma))}\|\bfeta\|_{L^2(0,T;\bH^{2+\kappa}(\Gamma))}\\
	&\leq C \|\bfeta\|^{\frac12}_{L^\infty(0,T;\bH^{2+\delta}(\Gamma))}\|\bfeta\|_{L^2(0,T;\bH^{2+\kappa}(\Gamma))}.\label{trace1}
\end{align}
This gives an estimate on the first term in the definition of $\mathcal{G}_m(\Gamma\times(0,T))$.
Now we focus on the second term. 
We observe that
\begin{align*}
	\|\Lambda^{2\kappa}\bfeta\|_{ H^1(0,T;\bH^{m-\frac52}(\Gamma))} \leq 	C\|\bfeta\|_{ H^1(0,T;\bH^{-1+\ep+2\kappa}(\Gamma))} \leq C\|\bfeta\|_{ H^1(0,T;\bH^{1+s}(\Gamma))}<C.
\end{align*}
Hence, by combining this observation with \eqref{JA} and by applying Theorem 8.2 in \cite{BH21}, we find
\begin{align*}
	\|\partial_t(J_\bfeta\nabla \bA_\bfeta^{-1}\Lambda^{2\kappa}\bfeta|_\Gamma)\|_{L^2(0,T;\bH^{m-\frac52}(\Gamma))} &\leq \|(J_\bfeta\nabla A_\bfeta^{-1})|_\Gamma\|_{L^\infty(0,T;\bH^{1+\ep}(\Gamma))}\|\Lambda^{2\kappa}\bfeta\|_{ H^1(0,T;\bH^{m-\frac52}(\Gamma))}\\
	& +  \|(J_\bfeta\nabla A_\bfeta^{-1})|_{\Gamma}\|_{H^1(0,T;\bH^{\ep}(\Gamma))}\|\Lambda^{2\kappa}\bfeta\|_{ L^\infty(0,T;\bH^{m-\frac32}(\Gamma))}\\
	& \leq  C\|\bfeta\|^{\frac12}_{L^\infty(0,T;\bH^{2+4\ep}(\Gamma))}\|\bfeta\|_{H^1(0,T;\bH^{1+s}(\Gamma))}\\
	&+C\|\partial_t\bfeta\|_{L^2(0,T;\bH^{1+s}(\Gamma))}^2\|\bfeta\|_{L^\infty(0,T;\bH^{\ep+2\kappa}(\Gamma))}.
\end{align*}
Since we took $\ep< \min\{\frac{\delta}4,1-\kappa\}$ and $\kappa<1$, we see that $\|\bfeta\|_{L^\infty(0,T;\bH^{\ep+2\kappa}(\Gamma))} \leq \|\bfeta\|_{L^\infty(0,T;\bH^{2}(\Gamma))} \leq C$. Hence we conclude that,
\begin{align}
	\|\bfvarphi|_\Gamma\|_{H^1(0,T;\bH^{m-\frac52}(\Gamma))} 	& \leq  C+C\|\bfeta\|^{\frac12}_{L^\infty(0,T;\bH^{2+\delta}(\Gamma))}.\label{trace2}
\end{align}
Finally, interpolating between the spaces in \eqref{trace1} and \eqref{trace2}, we obtain 
\begin{align}
	\|\bfvarphi|_\Gamma\|_{ H^{\frac{2m-1}{2m}}(0,T;\bH^{(1-\frac2m)(m-\frac12)}(\Gamma))}& \leq \|\bfvarphi|_\Gamma\|^{\frac{2m-1}{2m}}_{H^1(0,T;\bH^{m-\frac52}(\Gamma))} \|\bfvarphi|_\Gamma\|^{\frac1{2m}}_{L^2(0,T;\bH^{m-\frac12}(\Gamma))}\notag\\
	&\leq C+C \|\bfeta\|^{\frac12}_{L^\infty(0,T;\bH^{2+\delta}(\Gamma))}\|\bfeta\|_{L^2(0,T;\bH^{2+\kappa}(\Gamma))}.\label{trace3}
\end{align}

Finally, by combining all the estimates \eqref{trace1}, \eqref{trace2} and \eqref{trace3}, we arrive at the desired result \eqref{varphitrace}.
This completes the proof of Proposition~\ref{phi_trace_est}.
\end{proof}

This proposition implies that for $m=\frac32+\ep$ where $0<\ep<s-\delta \ll 1$, we can continue estimating the right hand-side of \eqref{reg_stokes} to obtain
\begin{align}\label{varphi1}
	\|\bfvarphi\|_{L^2(0,T;\bH^m(\sO))}+\|\partial_t \bfvarphi\|_{L^2(0,T;\bH^{m-2}(\sO))} \leq  C(1+ \|\bfeta\|^{\frac12}_{L^\infty(0,T;\bH^{2+\delta}(\Gamma))}\|\bfeta\|_{L^2(0,T;\bH^{2+\kappa}(\Gamma))}).
\end{align}

Moreover, thanks to the Lions-Magenes theorem, \eqref{varphi1} further gives us
\begin{align}\label{varphi2}
	\|\bfvarphi\|_{C(0,T;\bH^{\frac12+\ep}(\sO))}= 	\|\bfvarphi\|_{C(0,T;\bH^{m-1}(\sO))} \leq C(1+ \|\bfeta\|^{\frac12}_{L^\infty(0,T;\bH^{2+\delta}(\Gamma))}\|\bfeta\|_{L^2(0,T;\bH^{2+\kappa}(\Gamma))}),
\end{align}
which will ultimately be used in deriving bounds for the nonlinear term in the Navier-Stokes equations.

Now we turn our attention to deriving the relevant estimates for $\bq$. We will do so by using the relation \eqref{testfunction} and the estimates \eqref{varphi1} and \eqref{varphi2}.\\
First, since we have the embedding $H^\frac12(\sO) \hookrightarrow L^3(\sO)$, estimate \eqref{varphi1} gives us,
\begin{equation}
	\begin{split}\label{q1}
		\|\bq\|_{L^2(0,T;\bH^1(\sO))} &\leq 
		C\|\bA_\bfeta\|_{L^\infty(0,T;\bH^{2.5}(\sO))}\|\bfvarphi\|_{L^2(0,T;\bH^{\frac32}(\sO))}\\
		&\leq C\|\bfeta\|_{L^\infty(0,T;\bH^{2}(\Gamma))}(1+ \|\bfeta\|^{\frac12}_{L^\infty(0,T;\bH^{2+\delta}(\Gamma))}\|\bfeta\|_{L^2(0,T;\bH^{2+\kappa}(\Gamma))})\\
		&\leq  C\|\bfeta\|^{\frac12}_{L^\infty(0,T;\bH^{2+\delta}(\Gamma))}\|\bfeta\|_{L^2(0,T;\bH^{2+\kappa}(\Gamma))}.
	\end{split}
\end{equation}
However, to deal with the nonlinear term in the Navier-Stokes equations, this estimate will not be sufficient. Thus, using \eqref{varphi2} we arrive at the following estimate which is later used to find bounds for the terms $I_7$ and $I_8$:
\begin{equation}\label{q2}
	\begin{split}
		\|\bq\|_{L^\infty(0,T;\bH^{\frac12}(\sO))}& \leq C\|\bA_\bfeta\|_{L^\infty(0,T;\bH^{2.5}(\sO))}\|\bfvarphi\|_{L^\infty(0,T;\bH^{m-1}(\sO))}\\
		& \leq  C\|\bfeta\|_{L^\infty(0,T;\bH^{2}(\Gamma))}(1+ \|\bfeta\|^{\frac12}_{L^\infty(0,T;\bH^{2+\delta}(\Gamma))}\|\bfeta\|_{L^2(0,T;\bH^{2+\kappa}(\Gamma))})\\
		& \leq  C\|\bfeta\|^{\frac12}_{L^\infty(0,T;\bH^{2+\delta}(\Gamma))}\|\bfeta\|_{L^2(0,T;\bH^{2+\kappa}(\Gamma))}.
	\end{split}
\end{equation}
Finally, we expand:
\begin{align}\label{dtq}
	\partial_t\bq &= \partial_tJ_\bfeta\nabla \bA_\bfeta^{-1}\bfvarphi
	+ J_\bfeta\partial_t\nabla \bA_\bfeta^{-1}\bfvarphi
	+J_\bfeta\nabla \bA_\bfeta^{-1}\partial_t\bfvarphi.
\end{align}
We know that (see e.g. \cite{C88}), $$\partial_t J_{\bfeta} = -(J_{\bfeta} )^{-2}\text{tr}((\text{cof } \bA_\bfeta)^T \partial_t\nabla  \bA_\bfeta),$$ and thus handling of the first two terms on the right-hand side of \eqref{dtq} is straight-forward, as these terms remain bounded in $L^2(0,T;\bL^2(\sO))$. For the third term on the right-hand side of \eqref{dtq} we apply Theorem 8.1 in \cite{BH21}. Using \eqref{varphi1}, we observe that for 
$$m=\frac32+\ep, \ \ep<\frac{\delta}8, \ {\rm and\ any} \ q>\frac32$$
the following bound on $\bq$ holds true:
\begin{align}
	\|\partial_t\bq\|_{L^2(0,T;\bH^{m-2}(\sO))}&\leq C\|J_\bfeta\nabla \bA_\bfeta^{-1}\|_{L^\infty(0,T;\bH^{q}(\sO))}\|\partial_t\bfvarphi\|_{L^2(0,T;\bH^{m-2}(\sO))}\notag\\
	&\leq C\|\bfeta\|^2_{L^\infty(0,T;\bH^{2+\ep}(\Gamma))}\|\bfeta\|^{\frac12}_{L^\infty(0,T;\bH^{2+\delta}(\Gamma))}\|\bfeta\|_{L^2(0,T;\bH^{2+\kappa}(\Gamma))}\notag\\
	&\leq C\|\bfeta\|^\frac74_{L^\infty(0,T;\bH^{2}(\Gamma))}\|\bfeta\|^\frac14_{L^\infty(0,T;\bH^{2+8\ep}(\Gamma))}\|\bfeta\|^{\frac12}_{L^\infty(0,T;\bH^{2+\delta}(\Gamma))}\|\bfeta\|_{L^2(0,T;\bH^{2+\kappa}(\Gamma))}\notag\\
	& \leq C\|\bfeta\|^{\frac34}_{L^\infty(0,T;\bH^{2+\delta}(\Gamma))}\|\bfeta\|_{L^2(0,T;\bH^{2+\kappa}(\Gamma))}.
	\label{qt}
\end{align}
This completes the derivation of estimates for $\bq$ and $\bfpsi$ that we will use to bound the integrals $I_j$, $1\leq j\leq 11$ in \eqref{testing}.

We start discussing the bounds of integrals $I_j$, $1\leq j\leq 11$ in \eqref{testing} by noticing 
that since $\bfxi$ is smooth, the bounds for the first 4 terms on the right hand side follow straight from the energy estimates derived in \eqref{apriori}. That is,
\begin{align*}
	|I_1+...+I_4| \leq C,
\end{align*}
where $C>0$ depends only on the given data 
	$\bu_0,\bv_0,\bfeta_0$. Note that, this constant technically also depends on the norms $\|\partial_t\bfxi\|_{L^2(0,;\bL^2(\sO))}$ and $\|\bfxi\|_{L^\infty(0,T;\bC^2(\Gamma))}$ which, according to our choice, are equal to 1.

Next, we present the derivation of the estimates that require further explanation.
We begin with $I_5$. To estimate $I_5$  we will use \eqref{qt}.  
Since $H^r_0(\sO)=H^r(\sO)$ for any $r<\frac12$, we obtain the following estimate which holds for $\frac32<m<2$:
\begin{align}\label{I5}
	|I_5| &= |\int_0^T	\int_{\sO}J_\bfeta\bu\cdot\partial_t\bq| \leq \int_0^T\|J_\bfeta\bu\|_{\bH^{2-m}(\sO)} \|\partial_t\bq\|_{\bH^{m-2}(\sO)}\\
	\nonumber
	& \leq C\int_0^T\|J_\bfeta\|_{\bH^{\frac32}(\sO)} \|\bu\|_{\bH^{1}(\sO)}\|\partial_t\bq\|_{\bH^{m-2}(\sO)}\\ 
	\nonumber
	& \leq C\|\bA_\bfeta\|_{L^\infty(0,T;\bH^{\frac52}(\sO))} \|\bu\|_{L^2(0,T;\bH^{1}(\sO))}\|\partial_t\bq\|_{L^2(0,T;\bH^{m-2}(\sO))}\\ 
	\nonumber
	&\leq C\|\bfeta\|_{L^\infty(0,T;\bH^{2}(\Gamma))}\|\bu\|_{L^2(0,T;\bH^1(\sO))}\|\bfeta\|_{L^\infty(0,T;\bH^{2+\delta}(\Gamma))}^{\frac34}\|\bfeta\|_{L^2(0,T;\bH^{2+\kappa}(\Gamma))}\\
	\nonumber
	&\leq CK^8_1 + \frac1{16}\|\bfeta\|_{L^\infty(0,T;\bH^{2+\delta}(\Gamma))}^2 + \frac18\|\bfeta\|_{L^2(0,T;\bH^{2+\kappa}(\Gamma))}^2,
\end{align}
where $K_1$ is the constant from Lemma \ref{uest} that depends on $\alpha_1$, $\|\bA_\bfeta\|_{L^\infty(0,T;\bW^{1,\infty}(\sO))}$ and $\|(\bA_\bfeta)^{-1}\|_{L^\infty(0,T;\bW^{1,\infty}(\sO))}$. 

\begin{remark}
Note that the choice $\frac32<m<2$ plays an important role here as we use the duality between $H^r$ and $H^{-r}$ for $r<\frac12$ on the right hand-side in the first line of the estimate of $I_5$ above. 
\end{remark}

To estimate the nonlinear terms in $I_6$ we use \eqref{q2}, to obtain
\begin{align*}
	|I_6| &\leq C\int_0^T(\|\bu+\bw\|_{\bL^6(\sO)})\|\bu\|_{\bH^1(\sO)}\|\bq\|_{\bL^3(\sO)}\leq C\int_0^T
	\|\bu\|^2_{\bH^1(\sO)}\|\bq\|_{\bH^\frac12(\sO)} \\
	&\leq C
	\|\bu\|^2_{L^2(0,T;\bH^1(\sO))}\|\bq\|_{L^\infty(0,T;\bH^\frac12(\sO))} \\
	&\leq CK_1^8 + \frac1{16}\|\bfeta\|_{L^\infty(0,T;\bH^{2+\delta}(\Gamma))}^2 + \frac18\|\bfeta\|_{L^2(0,T;\bH^{2+\kappa}(\Gamma))}^2.
\end{align*}

Similarly,  using  \eqref{q2} we estimate $I_7$:
\begin{align*}
	|I_7|=	|\int_0^t\int_{\sO} \partial_tJ_\bfeta \bu\cdot\bq| &\leq \|\bw\|_{L^2(0,T;\bH^1(\sO))} \|\bu\|_{L^2(0,T;\bL^6(\sO))}\|\bq\|_{L^\infty(0,T;\bL^3(\sO))}\\
	&\leq CK^4_1 + \frac1{16}\|\bfeta\|_{L^\infty(0,T;\bH^{2+\delta}(\Gamma))}^2 + \frac18\|\bfeta\|_{L^2(0,T;\bH^{2+\kappa}(\Gamma))}^2.
\end{align*}

The symmetrized gradient integral $I_8$ is estimated using \eqref{q1} to obtain
\begin{align*}
	|I_8| &\leq \|\bu\|_{L^2(0,T;\bH^1(\sO))}\|\bq\|_{L^2(0,T;\bH^1(\sO))} \\
	&\leq CK_1^4 + \frac1{16}\|\bfeta\|_{L^\infty(0,T;\bH^{2+\delta}(\Gamma))}^2 + \frac18\|\bfeta\|_{L^2(0,T;\bH^{2+\kappa}(\Gamma))}^2.
\end{align*}
Hence, absorbing appropriate terms on the left hand side we obtain
\begin{align}\label{reg1}
	\frac14\|\bfeta\|^{2}_{L^\infty(0,T;\bH^{2+\delta}(\Gamma))}+\frac12\|\bfeta\|^2_{L^2(0,T;\bH^{2+\kappa}(\Gamma))} \leq K_2 + \|\bfeta_0\|^2_{\bH^{2+\delta}(\Gamma)},
\end{align}
where the constant $K_2$ depends on  $\|\bA_\bfeta\|_{L^\infty(0,T;\bW^{1,\infty}(\sO))}, \|(\bA_\bfeta)^{-1}\|_{L^\infty(0,T;\bW^{1,\infty}(\sO))}$ and $\alpha$ due to its dependence on $K_1$ from Lemma \ref{uest}. Recall here that for any $0<\delta<s$, we chose $\kappa=1-(s-\delta)$. 

{\bf Bootstrap argument:} We will next prove the estimate \eqref{reg_ale} i.e.  we will get rid of the dependence of $K_2$, appearing in the right-hand side of \eqref{reg1}, on the norm $\|\bA_\bfeta\|_{L^\infty(0,T;\bW^{1,\infty}(\sO))}$ and on $\alpha<\inf_{\sO\times(0,T)} J_\bfeta$. We do this by possibly shrinking the time length on which this desired estimate holds by using a bootstrap argument \cite[Propostion 1.21]{TaoDisspersive}. 
Observe that, if for a fixed $\alpha$ and some $C_0>0$ we have
\begin{align}\label{hypo}
	\|\bA_\bfeta\|_{L^\infty(0,T;\bW^{1,\infty}(\sO))} \leq 2C_0\quad\text{and}\quad \inf_{\sO \times(0,T)}J_{\bfeta}>\alpha,
\end{align}
then according to \eqref{reg1} there exists a constant $K_2>0$ depending on $C_0$ such that,
\begin{align*}
	\|\bfeta\|_{L^\infty(0,T;\bH^{2+\delta}(\Gamma))} \leq K_2.
\end{align*}
Furthermore, Sobolev and interpolation inequalities imply for any $\delta>0$ that
\begin{align}
	\|\bA_\bfeta-\bA_{\bfeta_0}\|_{L^\infty(0,T;\bW^{1,\infty}(\sO))} &\leq C \|\bA_\bfeta-\bA_{\bfeta_0}\|_{L^\infty(0,T;\bH^{\frac{5+\delta}2}(\sO))} \leq 	C\|\bfeta-\bfeta_0\|_{L^\infty(0,T;\bH^{2+\frac{\delta}{2}}(\Gamma))}\notag\\
	&\leq C\|\bfeta-\bfeta_0\|^{\frac{2+\frac{\delta}{2}}{2+\delta}}_{L^\infty(0,T;\bH^{2+\delta}(\Gamma))}\|\bfeta-\bfeta_0\|^{\frac{\frac{\delta}{2}}{2+\delta}}_{L^\infty(0,T;\bL^{2}(\Gamma))}\notag\\
	&\leq C(K_2)^{\frac{2+\frac{\delta}{2}}{2+\delta}} T^{\frac{\delta}{2(2+\delta)}}\| \partial_t\bfeta\|_{L^\infty(0,T;\bL^{2}(\Gamma))}^{\frac{\delta}{2(2+\delta)}}	\notag\\
	&\leq  C(K_2)^{\beta} C_1^{1-\beta} T^{1-\beta}, \qquad \text{where }\beta=\frac{2+\frac{\delta}{2}}{2+\delta},\label{smalltime}
\end{align}
where $C$ depends only on $\sO$, $K_2$ depends on $C_0$, and the constant $C_1$ appearing in \eqref{apriori} depends on the given data $\bu_0,\bv_0,\bfeta_0$. 

Similarly, 
\begin{align}
	\inf_{\sO\times(0,T)}  J_{\bfeta}(t)&\geq \inf_{\sO} J_0-\sup_{\sO\times(0,T)}| J_{\bfeta}-J_{\bfeta_0}|\notag\\
	&\geq \inf_{\sO}J_0-C\|\bA_\bfeta-\bA_{\bfeta_0}\|_{L^\infty(0,T;\bW^{1,\infty}(\sO))}\notag\\
	\label{JacBound}
	&\geq\inf_{\sO} J_0- C(K_2)^{\beta} C_1^{1-\beta} T^{1-\beta}.
\end{align}
Hence, for small enough $T_0
>0$, the hypothesis \eqref{hypo} then implies that
\begin{align}\label{conc}
	\|\bA_\bfeta\|_{L^\infty(0,T_0;\bW^{1,\infty}(\sO))} \leq C_0 \quad \text{ and} \quad \inf_{\sO \times(0,T_0)}J_{\bfeta}>2\alpha.
\end{align}

This concludes our bootstrap argument and thus the proof of Theorem \ref{thm:Lip1}.

Next ,we focus on the proof of Theorem \ref{thm:Lip2} to establish the estimate \eqref{reg_temp}.

\subsection{The temporal regularity result for fluid and structure velocities.}\label{subsec:temp}
In this section we prove Theorem \ref{thm:Lip2}.
Namely,  the aim of this section is to show that for $(\bu,\bfeta)$, a pair of smooth functions that solves \eqref{weaksol}, there exists a constant $C>0$, depending only on $\sO$ and the given data, such that the following estimate holds for any $h>0$:
\begin{align}\label{transl}
	\|\tau_h\bu-\bu\|_{L^2(h,T_0;\bL^2(\sO))} + \|\tau_h\partial_t\bfeta-\partial_t\bfeta\|_{L^2(h,T_0;\bL^2(\Gamma))} \leq Ch^{\frac18},
\end{align}
where $T_0$ is the time length appearing in Theorem \ref{thm:Lip1}.

To obtain the two terms on the left-hand side of the estimate above, we will construct an appropriate pair of test functions $(\bq,\bfpsi)$ for the weak formulation  \eqref{weaksol} on the fixed domain $\sO$. A typical approach to obtaining results of this kind for the weak solutions of Navier-Stokes equations posed on a fixed domain, is to use the time integral $\int_{t-h}^t$ of the solution as a test function. This approach cannot directly be employed in the case of moving boundary problems since the fluid velocity at different times is defined on different domains. Thus we face issues, due to the motion of the fluid domain, that arise due to the incompressibility condition and the kinematic coupling condition. Our plan, in the spirit of \cite{G08},
is to construct the desired test function by first modifying the solution $(\bu,\bv)$ appropriately and then integrating it from $t-h$ to $t$, for any $t\in[0,T]$ and $h>0$. This modification must preserve the divergence of the fluid velocity and its boundary behavior (i.e. the kinematic coupling condition) which is not trivial to find due to the time-varying domain.
Note also that, due to the mismatch between the spatial regularity of the structure velocity $\partial_t\bfeta$ and that of the test function $\bfpsi$ in the weak formulation Definition \ref{weakform_steady}, this modification must also include a construction of a spatially regularized version of the structure velocity so that its time integral can be used as a test function in \eqref{weaksol}. 

For the construction of the fluid test function, we will first extend the structure velocity $\bv:=\partial_t\bfeta$ in the steady fluid domain $\sO$, subtract it from $\bu$ and then integrate the resulting function from $t-h$ to $h$ against an appropriate kernel that possesses the desired property of preserving divergence in $\sO$ and flux across $\partial\sO$. Now, to balance out this extra term in the fluid test function, i.e. the extension of $\bv$, and to construct the test function $\bfpsi$ that has the desired spatial $H^2$-regularity, we add the extension of the structure velocity $\bv$ projected on a finite dimensional subspace of $\bH^2$ to the fluid test function. This finite-dimensional projection also enjoys nice properties that result in the second term appearing on the left-hand side of \eqref{reg_temp}.

Finally, due to the addition of these two extra terms (i.e. the extension of $\bv$ and that of its finite dimensional truncation) the transformed-divergence of the fluid test function has to be corrected. For that purpose, we construct a Bogovski-type operator on the physical moving domain with the aid of the Bogovski operator on {\it the fixed domain $\sO$}. This is crucial. Since, on the fixed domain, the test functions are required to satisfy the transformed-divergence-free condition, we correct it by multiplying it with the inverse of the cofactor matrix of $\bA_\bfeta$ and by using the Piola identity.
We will now give precise definitions of our construction.

We fix $h>0$. Let $P_M$ denote the orthonormal projector in $L^2(\Gamma)$ onto the space $\text{span}_{1\leq i\leq M}\{\varphi_i\}$, where $\varphi_i$ satisfies $-\Delta \varphi_i=\lambda_i\varphi_i$ and $\varphi_i=0$ on $\partial\Gamma$. For any $\bv\in\bL^2(\Gamma)$ we use the notation $\bv_M=P_M\bv$ (i.e. subscript $M$) where $P_M$ is the projection onto span$_{1\leq j\leq M}\{\varphi_j\}$. We know that $\lambda_M \sim M^2$ and hence we will choose
\begin{align*}
	\lambda_M = ch^{-\frac34}.
\end{align*}

We now construct a simple extension of $\bv:=\partial_t\bfeta$ in the fluid domain. Let $\bw = \bv \chi$ where $\chi$ is a cut-off function applied to $\bv$ so that it does not have any contributions at the boundary $\partial\sO$ i.e. $\chi$ is a function smooth in $\sO$ such that $\chi(x,y,1)= 1$ and $\chi(x,y,0)= 0$. 
Then, for any $t\in[0,T_0]$, we define our fluid test function as follows (see also \cite{T24slip}): 
\begin{align}
	\bq(t)&:= (J_{\bfeta}(t))^{-1}\nabla \bA_\bfeta(t)\int_{t-h}^t \left( J_\bfeta(s)(\nabla  \bA_\bfeta(s))^{-1} (\bu(s)-\bw(s))\right)ds \notag	\\
	&+ \int_{t-h}^t\left( \bw_M(s)-\left( \frac{b_M(s,t)}{b_0(s,t)}\right) \bfxi_0(s)\chi \right) ds \notag\\
	&	-(J_{\bfeta}(t))^{-1}\nabla \bA_\bfeta(t)\int_{t-h}^t\sB\Big( \text{div} \left( J_\bfeta(s)(\nabla  \bA_\bfeta(s))^{-1} \bw(s)-J_{\bfeta}(t)(\nabla  \bA_\bfeta(t))^{-1} \bw_M(s)\right)\notag\\
	&	+\left( \frac{b_M(s,t)}{b_0(s,t)}\right) \text{div}(J_{\bfeta}(t)\nabla \bA_\bfeta^{-1}(t)\bfxi_0(s)\chi)\Big)ds.\notag
\end{align}
In other words, we define the fluid test function $\bq$ as the time integral from $t-h$ to $t$ of a modification $\bu_M$ of $\bu$:
\begin{equation}
	\begin{split}\label{test_f}
		\bq(t)&:= \int_{t-h}^t \left[ P_t^{-1} P_s (\bu(s)-\bw(s))	+\left( \bw_M(s)- \mathsf{c}_M(s,t) \bfxi_0(s)\chi \right) \right] ds \\
		&	-\int_{t-h}^tP_t^{-1}\sB\Big( \text{div} \left( P_s\bw(s)-P_t \bw_M(s)	+\mathsf{c}_M(s,t) P_t\bfxi_0(s)\chi \right)\Big)ds\\
		&:= \int_{t-h}^t\bu_M(s,t)ds,
	\end{split}
\end{equation}
where:
\begin{itemize}
	\item $\bw=\bv\chi$ and $\bw_M=\bv_M\chi$, 
	\item	$\chi$ is a smooth function such that $\chi(x,y,1)= 1$ and $\chi(x,y,0)= 0$ for $(x,y)\in \Gamma$,
	\item $P_t({\bf f}):=((J_{\bfeta})\nabla \bA^{-1}_{\bfeta}\cdot {\bf f})(t)$ is the Piola transformation composed with the ALE map $\bA_\bfeta$ for any ${\bf f}\in\bL^2(\sO)$.
	\item $\sB$ is the Bogovski operator on the fixed domain $\sO$. Recall that (see e.g. \cite{GHH06, Galdi}) if $\int_\sO f=0$ then
	$$\nabla\cdot\sB(f)=f, \quad \text{and}\quad \|\sB(f)\|_{\bH^{1}_0(\sO)} \leq c\|{ f}\|_{L^2(\sO)}.$$
	\item The correction term $\mathsf{c}_M\bfxi_0\chi$ where $\mathsf{c}_M(s,t)= \frac{b(s,s)-b_M(s,t)}{b_0(s,t)} \in L^\infty([0,T]^2) $ ensures that the above condition $\int_\sO f=0$ is met.
	\item Here, $b_M(s,t)=\int_\Gamma(\nabla\bfeta(t) \times \bv_M(s))$, $b(s,t)=\int_\Gamma(\nabla\bfeta(t)\times\bv(s))$ and $b_0(s,t)=\int_\Gamma(\nabla\bfeta(t) \times \bfxi_0(s))$,
	\item $\bfxi_0 \in \bC^\infty_0(\Gamma\times[0,T_0])$ is chosen such that $b_0(s,t) = 1$ for any $s,t \in [0,T_0]$.
\end{itemize}
For the structure test function we define:
\begin{align}\label{test_s}
	\bfpsi(t):=\int_{t-h}^t\left( \bv_M(s)-
	\mathsf{c}_M(s,t)\bfxi_0(s)\right) ds.
\end{align}
We will first show that $\bq$ has the required regularity. Thanks to the embedding
$$W^{1,\infty}(0,T;\bL^2(\Gamma))\cap L^\infty(0,T;\bH^2(\Gamma)) \hookrightarrow C^{0,\theta}(0,T;\bH^{2-2\theta}(\Gamma)).$$
we have for any $\theta \in (0,1)$ and $m\geq 2-2\theta$, that
\begin{align}\label{holder_eta}
	\|\tau_h\bfeta-\bfeta\|_{L^\infty(0,T_0;\bH^{m}(\Gamma))}\leq \|\bfeta\|_{C^{0,\theta}(0,T_0;\bH^{2-2\theta}(\Gamma))} h^\theta \leq Ch^{\theta}.
\end{align}
Moreover, 
setting $s=t-h$ for any $t\in[h,T_0]$ we see that
\begin{align}
	|\mathsf{c}_M(s,t)| &\leq |b(s,s)-b_M(s,t)| =
	|\int_\Gamma (\nabla\bfeta(t)-\nabla\bfeta(s))\times\bv_M(s) + \nabla\bfeta(s)\times(\bv_M-\bv)(s)|\notag\\
	&	\leq \|\bfeta(t)-\bfeta(s)\|_{\bH^1(\Gamma)}\|\bv_M\|_{\bL^2(\Gamma)} + \|\bfeta(s)\|_{\bH^2(\Gamma)}\|\bv_M(s)-\bv(s)\|_{\bH^{-1}(\Gamma)}\notag\\
	& \lesssim h^{\frac12}\|\bfeta\|_{C^{0,\frac12}(0,T_0;\bH^1(\Gamma))} + \lambda_M^{-1}\|\bv(s)\|_{\bH^1(\Gamma)}\notag\\
	&\leq Ch^{\frac12}\label{est:c}.
\end{align}
This estimate along with \eqref{reg_eta} and Lemma \ref{uest} further gives us the necessary bounds for the test function $\bq$:
\begin{align}
	\|\bq\|_{L^\infty(0,T_0;\bH^{1}(\sO))} &\leq \sup_{0\le t\leq T_0}\left( \|P_t\|^2_{\bH^{2.5+\delta}(\sO)}\int_{t-h}^t(\|\bu+\bw\|_{\bH^1(\sO)} +\|\bw_M\|_{\bH^1(\sO)} + |\mathsf{c}_M(s,t)|)ds\right) \notag\\
	&\leq h^{\frac12} \|\bfeta\|^4_{L^\infty(0,T_0;\bH^{2+\delta}(\Gamma))} ( \|\bu\|_{L^2(0,T_0;\bH^1(\sO))}+\|\bw\|_{L^2(0,T_0;\bH^1(\sO))} )\notag\\
	&\leq Ch^{\frac12}.\label{boundsq}
\end{align}
Next, we readily observe that,
$$\bq|_\Gamma = \bfpsi, \qquad\text{ on } (0,T)\times\Gamma.$$
Moreover, thanks to Theorem 1.7-1 in \cite{C88}, we have that
\begin{equation}
	\begin{split}\label{stot}
		&	\text{div}^{\bfeta(t)}\bq_{M}(t)=
		{J_{\bfeta}(t)}^{-1}\int_{t-h}	^t\text{div}(P_s \bu(s))= {J_{\bfeta}(t)}^{-1}\int_{t-h}	^t{J_\bfeta(s)\text{ div}^{\bfeta(s)}\bu(s)}=0.
	\end{split}	
\end{equation}
Hence $(\bq_M,\bfpsi_M)$ is a valid test function for \eqref{weaksol}. For this pair of test functions we get:
\begin{equation}\label{weakform_reg}
	\begin{split}
		&-\int_h^{T_0}\int_\sO{ J}_{\bfeta}\bu\cdot \partial_t\left( \int_{t-h}^t{\bu_M}\right)  -\int_h^{T_0}\int_\Gamma \partial_t {\bfeta} \left( \partial_t \int_{t-h}^t\bv_M\right) = \int_h^{T_0 }\int_\Gamma \Delta\bfeta\cdot\Delta \bfpsi_M\\
		&+\gamma \int_h^{T_0 }\int_\Gamma \Lambda^{1+s}\partial_t\bfeta:\Lambda^{1+s} \bfpsi_M -\int_h^{T_0}\int_{\sO}J_{\bfeta}(\bu\cdot\nabla^{\bfeta }\bu \cdot\bq
		- \bw^{\bfeta}\cdot\nabla^{\bfeta }\bu\cdot\bq_M )\\
		&+\int_h^{T_0}\int_{\sO}\partial_tJ_\bfeta 
		\bu\cdot\bq_M 
		-2\nu\int_h^{T_0}\int_{\sO}{J}_{\bfeta}\, \bD^{{\bfeta} }(\bu ): \bD^{{\bfeta} }(\bq_M)\\
		&= I_1+...+I_5.
	\end{split}
\end{equation}
Before estimating each term on the right-hand side of the equation above, we observe that the first term on the left-hand side can be written as,
\begin{align*}
	\int_h^{T_0}\int_{\sO}J_\bfeta\bu&\cdot\partial_t\left( \int_{t-h}^t{\bu_M}\right) =	\int_h^{T_0}\int_{\sO}J_\bfeta\bu\cdot\left( \bu_M(t,t)-\bu_M(t-h,t)+\int_{t-h}^t\partial_t\bu_M(s,t)ds\right)\\
	&	=\int_h^{T_0}\int_{\sO}J_\bfeta\bu\cdot(\bu(t)-\bu(t-h))\\
	&	+\int_h^{T_0}\int_{\sO}J_\bfeta\bu\cdot\left( \bu_M(t,t)-\bu(t)-\bu_M(t-h,t)+\bu(t-h)+\int_{t-h}^t\partial_t\bu_M(s,t)ds\right)\\
	&= I_0^1+I_0^2+I_0^3.
\end{align*}

Observe that the term $I_0^1$ on the right-hand side gives us the desired term $\|\tau_h\bu-\bu\|_{L^2(0,T;\bL^2(\sO))}$ in the left-hand side of \eqref{transl} since we can write
\begin{align*}
	I_0^1&=	\int_h^{T_0}\int_{\sO}J_\bfeta(t)\bu(t)\cdot(\bu(t)-\bu(t-h)) \\
	&= \frac12\int_h^{T_0}\int_{\sO}J_\bfeta\left( |\bu(t)|^2 - |\bu(t-h)|^2+|\bu(t)-\bu(t-h)|^2\right) .
\end{align*}
The second term on the right side in the expression for $I_0^1$ above  is the desired term, whereas the first one can be bounded as follows:
\begin{align*}
	\int_h^{T_0}\int_{\sO}J_\bfeta&\left( |\bu(t)|^2 - |\bu(t-h)|^2\right) = 	\int_h^{T_0}\int_{\sO}J_\bfeta(t) |\bu(t)|^2 - J_\bfeta(t-h) |\bu(t-h)|^2\\
	&-\int_h^{T_0}\int_\sO(J_\bfeta(t)-J_\bfeta(t-h))|\bu(t-h)|^2\\
	&\leq h\left( \sup_{0\leq t\leq T_0}\int_{\sO}J_\bfeta(t)|\bu(t)|^2 + \|\partial_t J_\bfeta\|_{L^\infty(0,T_0;\bL^2(\sO))} \|\bu\|^2_{L^2(0,T_0;\bL^4(\sO))}\right) \\
	&\leq Ch.
\end{align*}
Before analyzing $I_0^2$ we treat the third term $I_0^3$. We observe that for any $t\in(0,T_0)$  we have
\begin{align*}
	\|\partial_t\bu_M(s,t)\|_{\bL^\frac32(\sO)}&\leq\|	\partial_tP_t^{-1}P_s(\bu(s)-\bw(s))\|_{\bL^\frac32(\sO)} + \sup_{t-h\leq s\leq t} \partial_t\mathsf{c}_M(s,t)\|	\bfxi_0(s)\chi\|_{\bL^2(\sO)}\\
	&\leq C(\|P_s\|_{\bL^\infty(\sO)}\|\partial_tP_t\|_{\bL^2(\sO)}\|\bu(s)-\bw(s)\|_{\bL^6(\sO)}+ \|\bv(t)\|_{\bH^1(\Gamma)}\|\bv_M(s)\|_{\bL^2(\Gamma)}).
\end{align*}
Hence, we arrive at the following estimate for the term $I_0^3$,
\begin{align*}
	|I_0^3|&=|\int_h^{T_0}\int_\sO J_\bfeta\bu	\left( \int_{t-h}^t\partial_t\bu_M(s,t)ds\right) d{\bf x}dt| \\
	&\leq h^{\frac12}\int_h^{T_0} \left( \|J_\bfeta(t)\|_{\bL^6(\sO)}\|\bu(t)\|_{\bL^3(\sO)}
	\left( \int_{t-h}^t\|\partial_t\bu_M(s,t)\|^2_{\bL^\frac32(\sO)}ds\right)^\frac12  \right) dt\\
	&\leq Ch^\frac12.
\end{align*}
Now we begin with our calculations for the term $I_0^2$. First note that
\begin{align*}
	\|\bw_M-\bw\|_{\bL^2(\sO)} \leq C\|\bv_M-\bv\|_{\bH^{-\frac12}(\Gamma)} \leq C\lambda_M^{-\frac34}\|\bv\|_{\bH^1(\Gamma)}.
\end{align*}
Moreover, thanks to the embedding $H^{\frac12}(\sO)\hookrightarrow L^3(\sO)$ and the fact that the ALE maps satisfy \eqref{ale_Hs}, the estimate \eqref{holder_eta} gives us,
\begin{align*}
	\sup_{t-h\leq s \leq t}	\|P_s-P_t\|_{\bL^3(\sO)}& \leq \|\bfeta\|_{L^\infty(0,T_0;\bH^{2+\delta}(\Gamma))}	\sup_{t-h\leq s \leq t}\|\bfeta(s)-\bfeta(t)\|_{\bH^1(\Gamma)} \\
	&\leq h^{\frac12}\|\bfeta\|_{C^{0,\frac12}(0,T_0;\bH^1(\Gamma))} \\
	&\leq Ch^{\frac12}.
\end{align*}
These two estimates put together give us,
\begin{align*}
	\int_h^{T_0}& 	\|\bu_M(t-h,t)-\bu(t-h)\|^2_{\bL^2(\sO)}dt\\
	&\leq \int_h^{T_0}(\|P^{-1}_t\|_{\bL^\infty(\sO)}\|P_{t-h}-P_t\|_{\bL^3(\sO)}\|\bu-\bw\|_{\bL^6(\sO)} + \|\bw-\bw_M\|_{\bL^2(\sO)} +|\mathsf{c}_M(t-h,t)|)^2dt\\
	& \leq Ch^{}.
\end{align*}
Hence, for the term $I_0^2$ we find,
\begin{align*}
	|I_0^2| &= |\int_h^{T_0}\int_{\sO}J_\bfeta\bu\cdot( \bu_M(t,t)-\bu(t)-\bu_M(t-h,t)+\bu(t-h))dxdt|\\
	&\leq\|J_\bfeta\|_{L^\infty(0,T_0;\bL^p(\sO))}\|\bu\|_{L^2(0,T_0;\bH^1(\sO))} (\int_h^{T_0}\|\bu_M(t-h,t)-\bu(t)\|^2_{\bL^2(\sO)}dt)^{\frac12} \\
	&\leq Ch^{\frac12}.
\end{align*}
Similarly, we treat the second term on the left-hand side of \eqref{weakform_reg} which will produce the second term on the left-hand side of our desired inequality \eqref{transl}. We write,
\begin{align*}
	\int_h^{T_0}\int_\Gamma \bv(t) \,\partial_t\left( \int_{t-h}^t{\bv_M}(s)ds\right)dt&=\int_h^{T_0}\int_\Gamma \bv(t) (\bv_M(t)-\bv_M(t-h) - (\mathsf{c}_M(t,t)-\mathsf{c}_M(t-h,t))\bfxi_0)\\
	&= J_0^1+ J_0^2.
\end{align*}
Observe that due to orthonormality of $P_M$ we have
\begin{align*}
	J_0^1&=\int_h^{T_0}\int_\Gamma \bv(t) (\bv_M(t)-\bv_M(t-h))dxdydt\\
	& = \int_h^{T_0}\int_\Gamma \bv_M(t) (\bv_M(t)-\bv_M(t-h))dxdydt \\
	&= \frac12\int_h^{T_0}\int_\Gamma ( |\bv_M(t)|^2-|\bv_M(t-h)|^2 + |\bv_M(t)-\bv_M(t-h)|^2)dxdydt.
\end{align*}
Here, as mentioned previously, the second term is another one of our desired terms in \eqref{transl} whereas the first term can be bounded as follows,
$$\int_h^{T_0} (\|\bv_M(t)\|_{\bL^2(\Gamma)}^2-\|\bv_M(t-h)\|_{\bL^2(\Gamma)}^2)dt =\int_{T_0-h}^{T_0}\|\bv_M(t)\|_{\bL^2(\Gamma)}^2dt\leq h\|\bv\|^2_{L^\infty(0,T_0;\bL^2(\Gamma))}.$$
Notice that here we also used the property of projection that states $\|P_M\bv\|_{\bH^k} \leq \gamma^{\frac{m-k}2}\|\bv\|_{\bH^m}$.
Thanks to \eqref{est:c}, we readily further deduce that
\begin{align*}
	|J_0^2|=|	\int_h^{T_0}\int_\Gamma (\mathsf{c}_M(t,t)-\mathsf{c}_M(t-h,t))\bfxi_0| \leq Ch^\frac12.
\end{align*}
This completes the treatment of the terms on the left-hand side of \eqref{weakform_reg}.
Hence, by combining all the estimates, we summarize that so far we have
\begin{align*}
	\|\tau_h\bu-\bu\|_{L^2(0,T_0;\bL^2(\sO))} + \|\tau_h\partial_t\bfeta-\partial_t\bfeta\|_{L^2(0,T_0;\bL^2(\Gamma))}	\leq Ch^{\frac12} + \sum_{j=1}^{8}|I_j|
\end{align*}

Now we estimate  the terms $I_j, j=1,...,5$ that appear on the right-hand side of \eqref{weakform_reg}. We start with $I_1$, which is one of the more crucial terms. Thanks to the property of projection operators stating $\|P_M\bv\|_{\bH^k} \leq \gamma^{\frac{m-k}2}\|\bv\|_{\bH^m}$ and the estimate \eqref{est:c}, we obtain
\begin{align*}
	|I_1|&=|\int_h^{T_0}\int_\Gamma \Delta\bfeta\cdot\Delta \bfpsi | =\int_h^{T_0}\int_\Gamma\Delta\bfeta\cdot\Delta \left( \int_{t-h}^t\left( \bv_M(s)-
	\mathsf{c}_M(s,t)\bfxi_0(s)\right) ds\right) \\
	&\leq \int_h^{T_0}\int_{t-h}^t\|\bfeta(t)\|_{\bH^2(\Gamma)}\left( \|\bv_M(s)\|_{\bH^2(\Gamma)}+ |\mathsf{c}_M(s,t)|\|\bfxi_0(s)\|_{\bH^2(\Gamma)}\right) dsdt\\
	& \leq \int_h^{T_0}\int_{t-h}^t\|\bfeta(t)\|_{\bH^2(\Gamma)}\left( \lambda_M\|\bv(s)\|_{\bL^2(\Gamma)}+|\mathsf{c}_M(s,t)|\|\bfxi_0(s)\|_{\bH^2(\Gamma)} \right) dsdt\\
	& \leq Ch\|\bfeta\|_{L^\infty(0,T_0;\bH^2(\Gamma))}(\lambda_M\|\bv\|_{L^\infty(0,T_0;\bL^2(\Gamma))}+h^{\frac12})\leq Ch^{\frac14}.
\end{align*}
Similarly, for the next term we obtain
\begin{align*}
	|I_2|=|\gamma \int_h^{T_0 }\int_\Gamma \Lambda^{1+s}\partial_t\bfeta:\Lambda^{1+s} \bfpsi| \leq h^{\frac12}	\|\bv\|_{L^2(0,T;\bH^{1+s}(\Gamma))} \leq Ch^{\frac12}.
\end{align*}
The next two terms $I_3,I_4$ are treated using \eqref{boundsq}. For the nonlinear term we have
\begin{align*} 
	|I_3|&= |\int_h^{T_0}\int_{\sO}J_{\bfeta}(\bu\cdot\nabla^{\bfeta }\bu \cdot\bq
	- \bw^{\bfeta}\cdot\nabla^{\bfeta }\bu\cdot\bq )|\\
	&\leq C\|J_\bfeta\|_{L^\infty(0,T_0;\bL^6(\sO))}\|\bu\|_{L^2(0,T_0;\bL^6(\sO))}\|\bu\|_{L^2(0,T_0;\bH^1(\sO))}\|\bq\|_{L^\infty(0,T_0;\bL^6(\sO))}\\
	&\leq C\|\bfeta\|_{L^\infty(0,T_0;\bH^2(\sO))}\|\bu\|^2_{L^2(0,T_0;\bH^1(\sO))}\|\bq\|_{L^\infty(0,T_0;\bH^1(\sO))}\\
	&\leq Ch^{\frac12}.
\end{align*} 
The terms $I_4$ and $I_5$ are treated identically. For $I_4$ we use the fact that $\bw^{\bfeta}$ is the harmonic extension of $\bv$ in $\sO$ which implies that $\|\partial_t\bA_\bfeta\|_{\bH^{k+\frac12}(\sO)} \leq C \|\bv\|_{\bH^k(\Gamma)}$. Hence \eqref{boundsq} leads to
\begin{align*}
	|I_4|=|	\int_h^{T_0}\int_{\sO}\partial_tJ_\bfeta \bu\cdot\bq| \leq \|\partial_tJ_\bfeta\|_{L^2(0,T;\bL^2(\sO))}\|\bu\|_{L^2(0,T;\bH^1(\sO))}\|\bq\|_{L^\infty(0,T;\bH^1(\sO))}\leq Ch^{\frac12}.
\end{align*}
Next, for $I_5$ we see that,
\begin{align*}
	|I_5|&=|\int_h^{T_0}\int_{\sO}{J}_{\bfeta}\, \bD^{{\bfeta} }(\bu ): \bD^{{\bfeta} }(\bq)| \leq C\|J_\bfeta\|_{L^\infty(0,T;\bL^\infty(\Gamma))}\|\bD^{{\bfeta} }(\bu )\|_{L^2(0,T;\bL^2(\sO))}\|\bD^{{\bfeta} }(\bq)\|_{L^2(0,T;\bL^2(\sO))}\\
	&\leq C\|J_\bfeta\|_{L^\infty(0,T;\bL^\infty(\Gamma))}\|(\nabla\bA_\bfeta)^{-1}\|^2_{L^\infty(0,T;\bL^\infty(\Gamma))}\|\nabla\bu \|_{L^2(0,T;\bL^2(\sO))}\|\nabla\bq\|_{L^2(0,T;\bL^2(\sO))}\\
	&\leq C\|\bfeta\|^3_{L^\infty(0,T;\bH^{2+\delta}(\Gamma))}\|\bu \|_{L^2(0,T;\bH^1(\sO))}\|\bq\|_{L^2(0,T;\bH^2(\sO))}\leq Ch^{\frac12}.
\end{align*}
Finally, we collect all the terms and obtain that,
\begin{align*}
	\|\tau_h\bu-\bu\|_{L^2(h,T_0;\bL^2(\sO))} + \|\tau_h\bv_M-\bv_M\|_{L^2(h,T_0;\bL^2(\Gamma))} \leq Ch^{\frac14}.
\end{align*}
Moreover, due to the following property of the projection $P_M$ $$\|\bv-P_M\bv\|_{L^2(0,T;\bL^2(\Gamma))} \leq C\lambda_M^{-\frac12}\|\bv\|_{L^2(0,T;\bH^1(\Gamma))} \leq Ch^{\frac38}\|\bv\|_{L^2(0,T;\bH^1(\Gamma))},$$
we come to our desired result,
\begin{align*}
	\|\tau_h\bu-\bu\|^2_{L^2(h,T_0;\bL^2(\sO))} + \|\tau_h\bv-\bv\|^2_{L^2(h,T_0;\bL^2(\Gamma))} \leq Ch^{\frac14}.
\end{align*}
This completes the proof of Theorem \ref{thm:Lip2}.

\section{Existence result}\label{sec:exist}

In this section we will provide a brief proof of Theorem~\ref{thm:exist}, namely, we prove the existence of a weak solution to our FSI problem. We will work with the problem posed on the fixed domain, 
and consider a weak formulation of the problem in the sense of Definition \ref{weakform_steady}. 
The plan is to discretize the problem in time and construct approximate solutions by employing the Lie operator splitting strategy to decouple the FSI problem into a fluid and a structure subproblem. This is done in the spirit of \cite{MC13} (see also \cite{MC14,MCG20}). For the convenience of the reader we will reproduce all the important steps in this section and for complementary details we refer to \cite{CKMT24}.
 We introduce the splitting scheme in the next subsection and then discuss the strategy for the rest of our proof in subsequent subsections.
	\subsection{The Lie operator splitting scheme}The overarching idea behind the Lie operator splitting scheme is to solve the evolution equation $\frac{d\phi}{dt} +S(\phi)=0$ by splitting the operator $S$ as a nontrivial sum $S=S_1+S_2$.  The time interval $(0,T)$ is divided into $N$ sub-intervals of size $\Delta t$ and on each sub-interval $(n\Delta t,(n+1)\Delta t)$ the evolution equations $\frac{d\phi_i}{dt} +S_i(\phi_i)=0$, $i=1,2$, are solved. 	In our case, we semidiscretize the problem in time and use operator splitting to divide the coupled problem {\it along the dynamic coupling condition} into two subproblems: a fluid and a structure subproblem.	The initial value for the structure sub-problem is taken to be the solution from the previous step, whereas the initial value for the fluid sub-problem is taken to be the just calculated solution found in the first sub-problem. 

 Our strategy is that in the first (structure) subproblem, we keep fluid velocity the same and update only the structure displacement and structure velocity, while in the second (fluid) subproblem we update the fluid and the structure velocities while keeping the structure displacement the same. The kinematic coupling condition is enforced in the second subproblem.

	For any $N\in\mathbb{N}$, we denote the time step by $\Delta t=\frac{T}{N}$ and use the notation $t^n=n\Delta t$ for $n=0,1,...,N$. 
	For any $N\in \mathbb{N}$ we introduce the following discrete total energy and dissipation for $i=0,1$ and $n=0,1,..,N-1$:
	\begin{equation}\label{EnDn}
		\begin{split}
			E_N^{n+\frac{i}2}&=\frac12\Big(\int_{\sO}J^n|\bu^{n+\frac{i}2}|^2 dx
			+\|\bv^{n+\frac{i}2}\|^2_{\bL^{2}(\Gamma)}+\|\bfeta^{n+\frac{i}2}\|^2_{\bH^2(\Gamma)}+\frac1N\|\bfeta^{n+\frac{i}2}\|^2_{\bH^3(\Gamma)}\Big),\\
			D_N^{n}&=\Delta t \int_{\sO} 2\nu J^n  |\bD^{\eta^{n}}(\bu^{n+1})|^2dx + \|\bv^{n+1}\|^2_{\bH^{1+s}(\Gamma)} .
	\end{split}\end{equation}
	
	The splitting scheme consisting of the two subproblems is defined as follows.  Let $(\bu^0,\bfeta^0,\bv^0)=(\bu_0,\bfeta_0,\bv_0)$ be the initial data. Then at the $j^{th}$ time level, we update the vector $(\bu^{n+\frac{j}{2}},\bfeta^{n+\frac{j}{2}},\bv^{n+\frac{j}{2}})$, where $j=1,2$ and $n=0,1,2,...,N-1$, according to the following scheme.
	
\noindent{\bf Structure sub-problem}: For any $n \leq N$ we look for $(\bfeta^{n+\frac12},\bv^{n+\frac12})$ such that
\begin{equation}
	\begin{split}\label{first}
		\bu^{n+\frac12}&=\bu^n,\\
		\int_\Gamma(\bfeta^{n+\frac12}-\bfeta^n) \bfphi &= (\Delta t)\int_\Gamma \bv^{n+\frac12}\bfphi ,\\
		\int_\Gamma \left( \bv^{n+\frac12}-\bv^n\right)  \bfpsi  &+ (\Delta t)\int_\Gamma \Delta\bfeta^{n+\frac12}\cdot\Delta \bfpsi + \frac{(\Delta t)}N\int_\Gamma \nabla^3\bfeta^{n+\frac12}\cdot\nabla^3 \bfpsi =0,
	\end{split}
\end{equation}
for any $\bfphi \in \bL^2(\Gamma)$ and $\bfpsi \in \bH^3(\Gamma)\cap \sV_S$.

\begin{remark} We notice that we have augmented this subproblem with a regularizing term, which is the last term on the left hand-side of the third equation
(the elastodynamics equation) in \eqref{first}. This term will vanish when we pass $N\to \infty$. The presence of this term is attributed to the fact that splitting the FSI problem along the dynamic coupling condition causes a mismatch between the structure velocity and the trace of the fluid velocity on $\Gamma$ (see also the second subproblem \eqref{second}). This is not ideal for an application of the regularity result of Theorem~\ref{thm:Lip1} at the level of approximate formulations (see Theorem \ref{thm:estN} below). This mismatch is taken care of by using bounds on numerical dissipation that result from the addition of this regularization term (see the discussion following \eqref{mismatch}). 
\end{remark}

Before commenting on the existence of a solution to the structure subproblem \eqref{first}, we introduce the second subproblem that updates the fluid and structure velocities in the fixed domain formulation \eqref{weaksol}.\\

\noindent {\bf	Fluid sub-problem}: Introduce the following functional space for the fluid velocity
$$\sV^n : = \{\bu\in \bH^1(\sO): \nabla^{\bfeta^n}\cdot \bu=0 \text{ on }\sO, \bu=0 \text{ on } \Gamma_r \}.$$
 We look for $(\bu^{n+1},\bv^{n+1}) \in \sV^n\times\bH^{1+s}(\Gamma)$ such that the following equations are satisfied for any $(\bq,\bfpsi)\in \sV^n\times\bH^{1+s}(\Gamma)$ such that $ \bq|_{\Gamma} =\bfpsi$:
\begin{equation}
	\begin{split}\label{second}
		&\qquad\qquad{\bfeta}^{n+1}:={\bfeta}^{n+\frac12}, \\
		&\int_{\sO}
		J^n\left( \bu^{n+1}-\bu^{n+{\frac12}}\right) \bq   +\frac{1}2\int_{\sO}
		\left( J^{n+1}-J^n\right)  \bu^{n+1}\cdot\bq \\
		&+\frac12(\Delta t)\int_{\sO}J^n((\bu^{n}-
		\bw^{{n+1}})\cdot\nabla^{{\bfeta}^n}\bu^{n+1}\cdot\bq - (\bu^{n}-
		\bw^{{n+1}})\cdot\nabla^{{\bfeta}^n}\bq\cdot\bu^{n+1}) \\
		&+2\nu(\Delta t)\int_{\sO}J^n \bD^{{\bfeta}^{n}}(\bu^{n+1})\cdot \bD^{{\bfeta}^{n}}(\bq) 
		+\int_\Gamma(\bv^{n+1}-\bv^{n+\frac12} )\bfpsi 
		\\
			& +(\Delta t)\int_\Gamma  \Lambda^{1+s}\bv^{n+1}\cdot \Lambda^{1+s}\bfpsi  	=0,
	\end{split}
\end{equation}
and the kinematic coupling condition is satisfied 
$$\bu^{n+1}|_{\Gamma}=\bv^{n+1}.$$ 
Here,  we use the notation
$$\bw^{n}=\frac{1}{\Delta t}(\bA_{{\bfeta}^{n+1}}-\bA_{{\bfeta}^n}),\quad J^n=\text{det}\nabla \bA_{{\bfeta}^n},$$
where $\bA_{\bfeta}$ denotes the solution to \eqref{ale} corresponding to the boundary data ${\bf id}+\bfeta$.

Equations \eqref{first} and \eqref{second} define the two steps in our splitting scheme. 

Next we discuss the existence of unique solutions for the two subproblems \eqref{first} and \eqref{second}.
\begin{theorem}[Existence and uniqueness result for the subproblems] 
\label{ExistenceSubProblems}
The following statements hold true:
	\begin{enumerate} 
		\item	Given $\bfeta^n\in \bH^2(\Gamma)$ and $\bv^{n}\in \bL^2(\Gamma)$ there exist unique $\bfeta^{n+\frac12},\bv^{n+\frac12}\in \bH^2(\Gamma)$ that solve \eqref{first}, and the following semidiscrete energy inequality holds:
	\begin{equation}\label{energy1}
		\begin{split}
			E^{n+\frac12} +\frac12\|\bv^{n+\frac12}-\bv^n\|_{\bL^2(\Gamma)}^2  +\frac12 \|\bfeta^{n+\frac12}-\bfeta^{n}\|_{\bH^2(\Gamma)}^2+\frac1{2N} \|\nabla^3\bfeta^{n+\frac12}-\nabla^3\bfeta^{n}\|_{\bL^2(\Gamma)}^2 \leq E^{n}.
		\end{split}
	\end{equation}
\item 	Given  $(\bu^{n+\frac12},\bv^n) \in \sV^{n-1}\times \bH^{1+s}(\Gamma)$ and $\bv^{n+\frac12}\in \bH^2(\Gamma)$, and $\bfeta^n\in \bH^2(\Gamma)$ assume  that $\inf_{\sO}J^n>\alpha$ for some fixed $\alpha>0$ for every $0\leq n\leq N$. 
Then there exists a unique $(\bu^{n+1},\bv^{n+1}) \in\sV^n\times \bH^{1+s}(\Gamma)$ that solves \eqref{second}, and the solution satisfies the following energy estimate
	\begin{equation}\label{energy2}
		\begin{split}
			E^{n+1}+D^{n}+\frac12\int_\sO J^n\left( |\bu^{n+1}-\bu^{n}|^2 \right)  &+\frac12\int_\Gamma|\bv^{n+1}-\bv^{n+\frac12}|^2 
			\leq E^{n+\frac12} ,
	\end{split}\end{equation}
where $E_n$ and $D_n$ are defined in \eqref{EnDn}.
\end{enumerate}
\end{theorem}
\begin{proof}
	The proof of this theorem involves an application of the Lax-Milgram Lemma in a way similar to 
	the proofs of  Propositions 1, 2, 3 and 4 in \cite{MC13}.
\end{proof}
The rest of the proof of Theorem \ref{thm:exist} can be divided into 3 parts: Constructing approximate solutions, finding uniform estimates for the approximate solutions and then passing $N\to\infty$ to prove that the limiting function is the desired solution, which involves a construction of appropriate test functions to be able to pass to the limit. We start with the construction of approximate solutions.
\subsection{Approximate solutions}\label{sec:approxsol}
In this subsection we will define two sequences of approximate solutions corresponding to the fluid velocity $\bu$, structure displacement $\bfeta$ and the structure velocity $\bv$. First, as is common with time-discretizations, we define the following approximations that are piece-wise constant in time: For $t \in (n\Delta t, (n+1)\Delta t]$ we let
\begin{align}\label{AppSol}
	\bu_{N}(t,\cdot)=\bu^{n+1},\quad \bfeta_{N}(t,\cdot)=\bfeta^{n+1},\quad \bv_{N}(t,\cdot)=\bv^{n+1},\quad \bv^*_N(t,\cdot)=\bv^{n+\frac12} .
\end{align}
Furthermore, we define the corresponding piecewise linear interpolations: for $t \in [t^n,t^{n+1}]$ we let  
\begin{equation}
	\begin{split}\label{approxlinear}
		&\tilde\bu_{N}( t,\cdot)=\frac{t-t^n}{\Delta t} \bu^{n+1}+ \frac{t^{n+1}-t}{\Delta t} \bu^{n}, \quad \tilde \bv_{N}(t,\cdot)=\frac{t-t^n}{\Delta t} \bv^{n+1}+ \frac{t^{n+1}-t}{\Delta t} \bv^{n}.
	\\	& \tilde\bfeta_{N}( t,\cdot)=\frac{t-t^n}{\Delta t} \bfeta^{n+1}+ \frac{t^{n+1}-t}{\Delta t} \bfeta^{n}.
	\end{split}
\end{equation}
Observe that
$$\partial_t\tilde\bfeta_N=\bv^*_N.$$

We now define $ \bA_{\bfeta_N}$ as the piecewise constant interpolations of the approximate ALE maps $\bA_{\bfeta^n}$. Observe that, by definition, $\bA_{\bfeta_N}$ solves \eqref{ale} with boundary value $\bfeta_N$ on $\Gamma$. We denote its Jacobian by $J_N:=\text{det}\nabla \bA_{\bfeta_N}$, which by definition is the piecewise constant interpolation of the functions $J^n$. We will also require piecewise linear interpolations of $J^n$ which we will denote by $\tilde J_N$. Along with that we also define the approximate ALE velocity $\bw_N$ to be the piecewise constant interpolation of $\bw^n$. Note that, by definition, $\bw_N$ solves \eqref{ale} with boundary data $\bv^*_N$. 

Using this notation, we combine the two subproblems \eqref{first} and \eqref{second} and then write the weak formulation satisfied by the approximate solutions in monolithic form as follows:
\begin{equation}\begin{split}\label{weakapprox}
&-\int_0^T\int_{\sO}( \tau_{\Delta t} J_N) \partial_t{\tilde\bu}_N\cdot \bq_N-\int_0^T\int_\Gamma \partial_t{\tilde \bv}_N\cdot\bfpsi_N=\frac12\int_0^T\int_\sO \partial_t\tilde J_N \bu_N\cdot\bq_N\\
&-\frac12\int_0^T\int_{\sO}(\tau_{\Delta t}{  J}_N )(( \tau_{\Delta t}\bu_N-
\bw_N)\cdot\nabla^{ \tau_{\Delta t}{\bfeta}_N } \bu_N\cdot\bq
_N- ( \tau_{\Delta t}\bu_N-
\bw_N)\cdot\nabla^{ \tau_{\Delta t}{\bfeta}_N }\bq_N\cdot \bu_N)\\
&-2\nu\int_0^T\int_{\sO}(\tau_{\Delta t}{  J}_N )\bD^{\tau_{\Delta t} {\bfeta}_N }( \bu_N)\cdot \bD^{\tau_{\Delta t} {\bfeta}_N }(\bq_N) -\int_0^T\int_\Gamma\Delta {\bfeta}_N\cdot\Delta \bfpsi_N -\frac1N\int_0^T\int_\Gamma\nabla^3 {\bfeta}_N\cdot\nabla^3 \bfpsi_N\\
&+\gamma\int_0^T\int_\Gamma  \Lambda^{1+s}\bv_N\cdot \Lambda^{1+s}\bfpsi_N,
	\end{split}
\end{equation}
for any $(\bq_N,\bfpsi_N)$ where $\bq_N(t) \in \sV_F^{\bfeta_N}(t)$ and $\bfpsi_N(t)\in \bH^{3}(\Gamma)$ satisfy $\bq_N|_{\Gamma}=\bfpsi_N$. 
Moreover, we have
\begin{align*}
&	\nabla^{\tau_{\Delta t}\bfeta_N}\cdot\bu_N=0,\qquad \bu_N|_{\Gamma}=\bv_N\\
& \bu_N(0,\cdot)=\bu_0,\,\, \bfeta_N(0,\cdot)=\bfeta_0,\,\,\bv_N(0,\cdot)=\bv_0.	
\end{align*}

In the subsequent sections we will show that these sequences are bounded independently of $N$ in certain appropriate spaces which will allow us to extract subsequences converging in weak and strong topologies of appropriate subspaces of the energy space. Using these convergence results we aim to  pass $N\to\infty$ in \eqref{weakapprox}.
\subsection{Uniform estimates}
In this section we will obtain the estimates, uniform in $N$, for the approximate solutions defined in Section \ref{sec:approxsol}.
\begin{theorem}\label{thm:est}
Assume, for some fixed $\alpha>0$, that $\inf_{\sO}J^n>\alpha$ for every $0\leq n\leq N$. Then there exists a constant $C>0$, independent of $N$ and $\ep$, such that
\begin{enumerate}
	\item $E^{n+1} \leq C, E^{n+\frac12} \leq C$, for every $n=0,1,..,N$.
	\item $\sum_{n=0}^{N-1}D^n \leq C$.	
	\item $\sum_{n=0}^{N-1}\left( 	\|\bv^{n+\frac12}-\bv^n\|_{\bL^2(\Gamma)}^2  + \|\bfeta^{n+\frac12}-\bfeta^{n}\|_{\bH^2(\Gamma)}^2+ \frac1{N} \|\bfeta^{n+\frac12}-\bfeta^{n}\|_{\bH^3(\Gamma)}^2\right)  \leq C$.
	\item $\sum_{n=0}^{N-1}\left( 	\int_\sO J^n\left( |\bu^{n+1}-\bu^{n}|^2 \right)  +\|\bv^{n+1}-\bv^{n+\frac12}\|^2_{\bL^2(\Gamma)}\right)  \leq C$,
\end{enumerate} 
where the discrete energy $E_N^n$ and dissipation $D_N^n$ are defined in \eqref{EnDn}.
\end{theorem}
\begin{proof}
	For a fixed $N\in\mathbb{N}$, we add the energy estimates for the two subproblems \eqref{energy1} and \eqref{energy2}, sum over $m\geq 1$, summing $0\leq n\leq m-1$ and then take supremum over $1 \leq m \leq N$. This gives us
	\begin{equation}\label{energysum}
		\begin{split}
			\sup_{1\leq m\leq N}E_N^m+\sum_{n=0}^{N-1}D_N^{n} +\sum_{n=0}^{N-1}C_N^{n}
			&\leq E^0,
	\end{split}\end{equation}
where,
$$	C^n_N:= \|\bv^{n+\frac12}-\bv^n\|_{\bL^2(\Gamma)}^2  + \|\bfeta^{n+\frac12}-\bfeta^{n}\|_{\bL^2(\Gamma)}^2 +
\int_\sO J^n\left( |\bu^{n+1}-\bu^{n}|^2 \right) +\|\bv^{n+1}-\bv^{n+\frac12}\|^2_{\bL^2(\Gamma)}.$$

\end{proof}
Next, we obtain uniform bounds for the approximate structure displacements and fluid velocity.
\begin{theorem}\label{thm:estN}
There exists $T_0>0$ such that for any $0<\delta<s$, 
\begin{enumerate}
	\item The sequences $\{{\bfeta}_{N}\}, \{{\tilde\bfeta}_{N}\}$ are bounded, independently of $N$, in $ L^\infty(0,T_0;\bH^{2+\delta}(\Gamma))$.
	\item The sequence $		\{\bA_{\bfeta_N}\}$ is bounded, independently of $N$ in ${L^\infty(0,T_0;\bC^{1,\delta}(\bar\sO))}$ and for some $\alpha>0$,
	the sequence of approximate Jacobians satisfies $\inf_{\sO \times(0,T)}J_{\bfeta_N}>\alpha$,  for all $N$.
	\item The sequence $\{\bu_{N}\}$ is bounded, independently of $N$, in 	$L^2(0,T_0;\bH^1(\sO))\cap L^\infty(0,T_0;\bL^2(\sO)).$ 
\end{enumerate}
\end{theorem}
\begin{proof}
We begin by observing that since $\tilde\bfeta_N$ and $\bfeta_N$ belong to $L^\infty(0,T;\bH^3(\Gamma))$, until some time, depending on $N$, the condition $\inf_{\sO\times(0,T)}J_N >\alpha$ must be satisfied.
We define ${T}_{N}^{max}$ to be maximal interval on which this lower bound for the approximate Jacobian  exists and hence the approximate solutions $(\bu_N,\tilde\bfeta_N)$ are defined, i.e. ${T}_{N}^{max}$ is a maximal time-interval such that assumptions of Theorem \ref{ExistenceSubProblems} hold. 
	
	First, we prove Statement (1). The proof of Statement (1) will rely heavily on the results of Theorem \ref{thm:Lip1}. 
	We first find uniform bounds for the structure displacement until time $T^{max}_N:=n^{max}_N(\frac{T}N)>0$, for some $n^{max}_N\in\mathbb{N}$ dependent on $N$. 
	For this purpose we reproduce the construction of the test functions and other important details from Theorem \ref{thm:Lip1} for the semi-discrete case.
	We take 
	\begin{align}\label{Lip_testN}
		\bq_N = -J_{\bfeta_N}^{-1}\nabla \bA_{\bfeta_N}\,\bfvarphi_N,\qquad \bfpsi_N = -(\Delta^\kappa\bfeta_N - c_N\bfxi_N),
	\end{align}
	where
	$\bfvarphi_N$ is the solution of \eqref{stokes} with boundary data $ \bfvarphi|_{\Gamma}  = J_{\bfeta_N} (\nabla \bA_{\bfeta_N})^{-1}|_{\Gamma}\left( \Delta^\kappa\bfeta_N - c_N\bfxi_N\right)$, where
		\begin{align}
		{c}_N
		= {\int_{\Gamma}\nabla\bfeta_N \times \Delta^\kappa\bfeta_N},
	\end{align}
	and $\bfxi_N$ is a smooth function satisfying ${\int_{\Gamma}\nabla\bfeta_N\times\bfxi_N}=1$ for every $t\in[0,T]$.

We use this test pair $(\bq_N,\bfpsi_N)$, defined in \eqref{Lip_testN}, in the weak formulation \eqref{weakapprox} on the time interval $(0,T^{max}_N)$ and follow, with some modifications, the steps presented in the proof of Theorem \ref{thm:Lip1}.
	
Observe that, $\bv_N$ is not equal to the structure velocity $\partial_t\tilde\bfeta_N$. Due to this mismatch caused by time-discretization and splitting, the term $\int_0^{T^{max}_N}\int_\Gamma  \Lambda^{{1+s}}\bv_N\cdot \Lambda^{1+s}\bfpsi_N $, appearing in \eqref{weakapprox} requires explanation. To that end, we write
	\begin{equation}
\begin{split}\label{mismatch}
	&	\gamma\int_0^{T^{max}_N}\int_\Gamma \Lambda^{1+s}\bv_N\cdot	\Lambda^{1+s}\bfpsi_N=	\gamma\int_0^{T^{max}_N}\int_\Gamma  \Lambda^{{1+s}}\partial_t\tilde\bfeta_N\cdot \Lambda^{{1+s}+2\kappa}\tilde\bfeta_N	\\
	&+	\gamma\int_0^{T^{max}_N}\int_\Gamma  \Lambda^{{1+s}}(\bv_N-\bv^*_N)\cdot  \Lambda^{{1+s}}\bfpsi_N		+\gamma\int_0^{T^{max}_N}\int_\Gamma  \Lambda^{{1+s}}\bv^*_N\cdot  \Lambda^{1+s+2\kappa}(\bfeta_N-\tilde\bfeta_N).
\end{split}
 \end{equation}
Observe that the first term on the right-hand side produces the desired $L^\infty(0,T^{max}_N;\bH^{2+\delta}(\Gamma))$-norm of $\tilde\bfeta_N$ in Statement (1). We will now show that the remaining two terms are bounded.
Thanks to integration-by-parts and the bounds on numerical dissipation obtained in Theorem \ref{thm:est}, we obtain
\begin{align*}
|	\gamma\int_0^{T^{max}_N}\int_\Gamma  \Lambda^{{1+s}}(\bv_N-\bv^*_N)\cdot  \Lambda^{{1+s}}\bfpsi_N	|&= |	\gamma\int_0^{T^{max}_N}\int_\Gamma  (\bv_N-\bv^*_N)\cdot  \Lambda^{2+2s}\bfpsi_N|\\
&\leq \|\bv_N-\bv^*_N\|_{L^2(0,T;\bL^2(\Gamma))}\|\bfeta_N\|_{L^2(0,T;\bH^3(\Gamma))}	\\
&\leq\left(  (\Delta t)\sum_{n=0}^{n^{max}_N-1}	\|\bv^{n+\frac12}-\bv^n\|_{\bL^2(\Gamma)}^2 \right) ^{\frac12}\|\bfeta_N\|_{L^2(0,T;\bH^3(\Gamma)).}
\end{align*}
Similarly, to bound the second term on the right hand side of \eqref{mismatch}, we first note, using Theorem \ref{thm:est} (3), that
\begin{align*}	\int_0^{T^{max}_N}\|\bfeta_{N}-\tilde\bfeta_{N}\|^2_{\bH^3(\Gamma)}dt
	&=
	\sum_{n=0}^{n^{max}_N-1}\int_{t^n}^{t^{n+1}}\frac1{\Delta t}\|(t-t^n)\bfeta^{n+1}+(t^{n+1}-t-\Delta t)\bfeta^{n}\|^2_{\bH^3(\Gamma)}dt\\
	&=\sum_{n=0}^{n^{max}_N-1}\|\bfeta^{n+1}-\bfeta^n\|_{\bH^3(\Gamma)}^2\int_{t^n}^{t^{n+1}} \left( \frac{t-t^n}{\Delta t}\right) ^2dt\\
	&\leq \frac{CT}{ N}\cdot N =CT.
\end{align*}
This gives us,
\begin{align*}
|	\gamma\int_0^{T^{max}_N}\int_\Gamma  \Lambda^{{1+s}}\bv^*_N\cdot  \Lambda^{1+s+2\kappa}(\bfeta_N-\tilde\bfeta_N)|& \leq \|\bv^*_N\|_{L^2(0,T^{max}_N;\bH^{1+s}(\Gamma))}\|\bfeta_N-\tilde\bfeta_N\|_{L^2(0,T^{max}_N;\bH^3(\Gamma))}\\
&\leq C.
\end{align*}
The rest of the terms in the formulation \eqref{weakapprox} with the test function \eqref{Lip_testN} follow the bounds obtained in the proof of Theorem \ref{thm:Lip1}.
Hence, the proof leading up to \eqref{reg1}, gives us that
\begin{align}
	\|\tilde\bfeta_N\|^{2}_{L^\infty(0,{ T^{max}_N};\bH^{2+\delta}(\Gamma))} +	\|\bfeta_N\|^{2}_{L^2(0,T^{max}_N;\bH^{3-(s-\delta)}(\Gamma))} \leq K_2 + \|\bfeta_0\|^2_{\bH^{2+\delta}(\Gamma)},\label{reg2}
\end{align}
where $K_2$ depends on $\|\bA_{\bfeta_N}\|_{L^\infty(0,T^{max}_N;\bW^{1,\infty}(\sO))}$ and $\inf_{\sO\times(0,T^{max}_N)}J_{\bfeta_N}$.
Now we use the continuity argument to get rid of this dependence of $K_2$ on the LHS norms by possibly reducing the 
length of the time interval. 

 Let $C_0>0$ and $\alpha >0$  be such that $\|\bA_{\bfeta_0}\|_{\bW^{1,\infty}(\sO)}<2C_0$ and $\inf_\sO J_0>\alpha$.
We hypothesize that
\begin{align}\label{hypo2}
	\|\bA_{\tilde\bfeta_N}\|_{L^\infty(0,T;\bW^{1,\infty}(\sO))} \leq 2C_0 \quad\text{and}\quad \inf_{\sO \times(0,T)}J_{\tilde\bfeta_N}>\alpha.
\end{align}
 Now we choose $T_0$ such that expression on the right hand side of \eqref{smalltime} is smaller then $C_0$, and the expression on the RHS of \eqref{JacBound} greater then $2\alpha$. We define,
$$T_N=\min\{T_0,{T}_{N}^{max}\}.$$
Then, thanks to the Sobolev embeddings used in \eqref{smalltime}, we conclude that 
\begin{align}\label{conc2}
	\|\bA_{\tilde\bfeta_N}\|_{L^\infty(0,{ T_N};\bW^{1,\infty}(\sO))} \leq C_0 \quad\text{and}\quad \inf_{\sO \times(0,T_N)}J_{\tilde\bfeta_N}>2\alpha.
\end{align}
Now, owing to the fact that by construction $\tilde\bfeta_N$ and thus $\bA_{\tilde\bfeta_N}$ are continuous in time, we can show that the conditions (a)-(d) of the Bootstrap principle  \cite[Propostion 1.21]{TaoDisspersive} are satisfied. Hence we have proven that, for any $0<\delta<s$, the sequence $\{{\tilde\bfeta}_{N}\}$ is bounded, independently of $N$, in $ L^\infty(0,T_{N};\bH^{2+\delta}(\Gamma))$.

To finish the proof of Statement (1) in Theorem \ref{thm:estN}, we show that
	there exists $N_0$ such that
	\begin{align}\label{timelowerbound}
		T_{N}=T_0 \quad\text{for all}\quad N\geq N_0.
	\end{align}
We prove \eqref{timelowerbound} by contradiction. Assume that \eqref{timelowerbound} is not true.
Recall that, by \eqref{conc2} we have that $\inf_{\sO} J_{\tilde\bfeta_N}(t)>2\alpha$ for every $t\in [0,T_N^{max}]$ where $T_N^{max}=n_N^{max}\Delta t$. 
However, then for small enough $\Delta t$ we can prolong the approximate solution, i.e. we can obtain $J^{n_N^{max}+1}>\alpha$ which contradicts maximality of $T_N^{max}$. 

Hence we have now shown that, for any $0<\delta<s$ the sequence $\{{\tilde\bfeta}_{N}\}$ is bounded, independently of $N$, in $ L^\infty(0,T_{0};\bH^{2+\delta}(\Gamma))$.

The rest of the statements follow thanks to \eqref{ale_Hs} and Lemma \ref{uest}.
\end{proof}

Thanks to Theorems \ref{thm:est} and \ref{thm:estN}, we can immediately conclude that there exist $\bfeta\in C(0,T;\bH^{2+\delta}(\Gamma)) \cap H^1(0,T;\bH^{1+s}(\Gamma))$ for some $0<\delta<s$, $\bu\in L^\infty(0,T;\bL^2(\sO))\cap L^2(0,T;\bH^1(\sO))$ and $\bv\in L^\infty(0,T;\bL^2(\Gamma))\cap L^2(0,T;\bH^{1+s}(\Gamma))$ such that the following weak and weak$^*$ convergence results hold, up to a subsequence, as $N\to \infty$:
\begin{enumerate}
	\item ${\bfeta}_{N} \rightharpoonup \bfeta $ weakly in  $ L^\infty(0,T_0;\bH^{2+\delta}(\Gamma))$ for any $0<\delta<s$.
	\item $\tilde\bfeta_{N} \rightharpoonup \bfeta$ weakly in  $C(0,T_0;\bH^{2+\delta}(\Gamma))$ for any $0<\delta<s$.
	\item $\bu_{N} \rightharpoonup \bu$ weakly in 	$L^2(0,T_0;\bH^1(\sO))$ and weakly$^*$ in $L^\infty(0,T_0;\bL^2(\sO)).$
	\item $\bv_{N} \rightharpoonup \bv$ weakly in  $ L^2(0,T_0;\bH^{1+s}(\Gamma))$ and weakly$^*$ in $L^\infty(0,T_0;\bL^2(\Gamma))$.
	\item $\bv^*_{N} \rightharpoonup \bv$ weakly$^*$ in $L^\infty(0,T_0;\bL^2(\Gamma))$.
\end{enumerate}  
Furthermore, 
$$\partial_t\bfeta=\bv\quad a.e. \text{ in } \sO\times(0,T).$$

We now seek to upgrade these results to strong convergence results to be able to pass to the limit in our nonlinear problem. 
In particular,  due to the geometric nonlinearity introduced in the fluid equations via the ALE maps associated with the motion of the fluid domain and its boundary,
we will require stronger convergence result for the structure displacements to be able to pass $N\to\infty$ in the approximate weak formulation.
Thus, we start with the following result on strong convergence of structure displacements.

\begin{proposition}\label{prop:strongeta}
There exists a  subsequence $\{\bfeta_N\}$ of approximate structure displacements  such that 
 \begin{align}\label{strongeta}
 	\bfeta_N \to \bfeta\quad  \text{ strongly in } L^\infty(0,T_0;\bH^{2+\delta}(\Gamma))\quad\text{ for any }  0\leq \delta<s .
 \end{align}
\end{proposition}
\begin{proof}
We begin by recalling the Aubin-Lions compactness lemma which states that the following embedding is compact
$$L^\infty(0,T;H^{2+\ep}(\Gamma))\cap W^{1,\infty}(0,T;L^2(\Gamma)) \subset\subset C([0,T];H^{2+\delta}(\Gamma)),\quad \text{for any }0\leq \delta<\ep. $$
Due to the uniform boundedness of $\tilde\bfeta_N$ in $L^\infty(0,T_0;\bH^{2+\ep}(\Gamma))$, for appropriately small $
\ep>0$ (see Theorem \ref{thm:est} (1)), and the uniform boundedness of $\bv^*_N=\partial_t\tilde\bfeta_N$ in $L^\infty(0,T_0;\bL^2(\Gamma))$, the compact embedding above implies that the sequence
\begin{align}
	\tilde\bfeta_N \to \bfeta \text{  strongly in } C([0,T_0];\bH^{2+\delta}(\Gamma)),\quad\text{ for any }  0\leq \delta<s .
\end{align}
Then, by comparing the definitions \eqref{AppSol} and \eqref{approxlinear} we conclude that \eqref{strongeta} holds for the sequences $\{\bfeta_N\}$ and $\{\tau_{\Delta t}\bfeta_N\}$.
\end{proof}
The consequences of the strong convergence \eqref{strongeta} in regards to approximate ALE maps are summarized in the following Proposition.

\begin{proposition}\label{ALEconv}
The ALE maps, defined by \eqref{ale}, satisfy the following strong convergence properties:
\begin{align}\label{ale_conv1}
	\bA_{ \bfeta _N} \to \bA_{ \bfeta } \ {\rm in} \  {L^\infty(0,T_0;\bW^{2,\frac{3}{1-\delta}}(\sO))} (\sO)),
	\\
	\label{ale_conv2}
	\nabla \bA_{\bfeta_N} \rightarrow \nabla \bA_{\bfeta}  \quad \text{ in }  L^\infty(0,T_0;\bC(\sO)),\\
	\label{ale_conv3}
	(\nabla \bA_{\bfeta_N})^{-1} \rightarrow (\nabla \bA_{\bfeta})^{-1} \quad \text{ in }  L^\infty(0,T_0;\bC(\sO))\\
	\label{ale_conv4}
			 J_N \rightarrow  J_{\bfeta}=\text{det}\nabla\bA_\bfeta,\quad \text{ in } L^\infty(0,T_0;C(\sO)).
\end{align}
Furthermore, let $\bw_\bfeta=\partial_t\bA_\bfeta$ be the solution of \eqref{ale} with respect to the boundary data $\bv$. Then:
\begin{align}\label{ale_conv5}
\bw_N \to \bw_\bfeta \quad \text{ in } L^2(0,T_0;\bH^\frac32(\sO)),
\\
\label{ale_conv6}
\partial_t\tilde J_N \to \partial_tJ_\bfeta = J_\bfeta(\nabla^\bfeta\cdot\bw_\bfeta)\quad \text{ in } L^2(0,T;L^3(\sO)).
\end{align}
\end{proposition}
\begin{proof}
First observe that due to the linearity of \eqref{ale}, the bounds \eqref{boundsA1} and Proposition \ref{prop:strongeta}, we have the following estimate,
which implies strong convergence \eqref{ale_conv1}:
\begin{align}\label{convale}
	\|\bA_{ \bfeta _N}-\bA_{ \bfeta }\|_{L^\infty(0,T_0;\bW^{2,\frac{3}{1-\delta}}(\sO))} \leq C	\|{ \bfeta _N}-{ \bfeta }\|_{L^\infty(0,T_0;\bH^{2+\delta}(\sO))} \to 0,
\end{align}
where  $\bA_\bfeta$ solves \eqref{ale} with respect to the boundary data $\bfeta$.

Estimate \eqref{convale}, along with  Proposition \ref{prop:strongeta}, imply the strong convergence results  \eqref{ale_conv2} - \eqref{ale_conv4},
as well as  \eqref{ale_conv5}.

To prove  \eqref{ale_conv6} we recall that for two matrices $A$ and $B$, the derivative of the determinant of $A$ acting on matrix $B$, denoted by
$D(det)(A)B$, is given by
$D(det)(A)B=det(A)tr(BA^{-1})$. Hence, by applying the mean value theorem to $det(\nabla \bA_{\bfeta^n})-det(\nabla \bA_{\bfeta^{n+j}})$
 we obtain, for some $\beta\in[0,1]$, that
\begin{equation}
\begin{split}\label{boundJt}
\left|\frac{J^{n+1}-J^{n}}{\Delta t} \right|&= \left|\text{det}(\nabla \bA^{n,\beta})\nabla^{n,\beta}\cdot\left( \frac{\bA_{\bfeta^{n+1}}-\bA_{\bfeta^{n}}}{\Delta t}\right) \right|,
\end{split}
\end{equation}
where $\nabla^{n,\beta}=\nabla^{\bfeta^n} +\beta(\nabla^{\bfeta^{n+1}}-\nabla^{\bfeta^n})$ and $\nabla \bA^{n,\beta}=\nabla \bA_{\bfeta^n} +\beta(\nabla \bA_{\bfeta^{n+1}}-\nabla \bA_{\bfeta^n})$. The details of these calculations can be found in \cite{MC16} (cf. (73)). Thus, \eqref{ale_conv1}-\eqref{ale_conv5} give us that
\begin{align*}
	\partial_t\tilde J_N \to \partial_tJ_\bfeta = J_\bfeta(\nabla^\bfeta\cdot\bw_\bfeta),\quad \text{ in } L^2(0,T;L^3(\sO)).
\end{align*}
This completes the proof Proposition~\ref{ALEconv}.
\end{proof}

\begin{remark}
We note that since $\bfeta\in L^\infty(0,T_0;\bH^{2+\delta}(\Gamma))$, we have that $\bA_\bfeta\in L^\infty(0,T_0;\bC^{1,\delta}(\sO))$. 
	Hence, on some time interval, still denoted by $(0,T_0)$, $\bA_\bfeta(t)$ is a diffeomorphism from $\sO$ to $\sO_\bfeta(t)$ for every $t\in(0,T_0)$. Recall, from Remark \ref{equiv}, that this is necessary to ensure the equivalence of Definitions \ref{weakform_moving} and \ref{weakform_steady} on $(0,T_0)$.
\end{remark}


\if 1= 0
	\begin{equation}\label{convrest}
		\begin{split}
			& \nabla \bA_{\bfeta_N} \rightarrow \nabla \bA_{\bfeta} \text{ and } (\nabla \bA_{\bfeta_N})^{-1} \rightarrow (\nabla \bA_{\bfeta})^{-1} \quad \text{ in }  L^\infty(0,T_0;\bC(\sO)),\\
			& J_N \rightarrow  J_{\bfeta}=\text{det}\nabla\bA_\bfeta,\quad \text{ in } L^\infty(0,T_0;C(\sO)),\\
			& \bw_N \to \bw_\bfeta, \quad \text{ in } L^2(0,T_0;\bH^\frac32(\sO)),
		\end{split}
	\end{equation}
	where $\bw_\bfeta=\partial_t\bA_\bfeta$ solves \eqref{ale} with respect to the boundary data $\bv$.
	
	Moreover, since $\bfeta\in L^\infty(0,T;\bH^{2+\delta}(\Gamma))$, we have that $\bA_\bfeta\in L^\infty(0,T;\bC^{1,\delta}(\sO))$. 
	Hence, on some time interval, still denoted by $(0,T_0)$, $\bA_\bfeta(t)$ is a diffeomorphism from $\sO$ to $\sO_\bfeta(t)$ for every $t\in(0,T_0)$. Recall, from Remark \ref{equiv}, that this is necessary to ensure the equivalence of Definitions \ref{weakform_moving} and \ref{weakform_steady} on $(0,T_0)$.
	
		Before moving on we comment on the convergence of the time derivative of the approximate Jacobian. We recall that for two matrices $A$ and $B$, the derivative of the determinant $D(det)(A)B=det(A)tr(BA^{-1})$. Hence, by applying the mean value theorem to $det(\nabla \bA_{\bfeta^n})-det(\nabla \bA_{\bfeta^{n+j}})$
 we obtain, for some $\beta\in[0,1]$, that
\begin{equation}
\begin{split}\label{boundJt}
|\frac{J^{n+1}-J^{n}}{\Delta t}|&= |\text{det}(\nabla \bA^{n,\beta})\nabla^{n,\beta}\cdot\left( \frac{\bA_{\bfeta^{n+1}}-\bA_{\bfeta^{n}}}{\Delta t}\right) |,
\end{split}
\end{equation}
where $\nabla^{n,\beta}=\nabla^{\bfeta^n} +\beta(\nabla^{\bfeta^{n+1}}-\nabla^{\bfeta^n})$ and $\nabla \bA^{n,\beta}=\nabla \bA_{\bfeta^n} +\beta(\nabla \bA_{\bfeta^{n+1}}-\nabla \bA_{\bfeta^n})$. The details of these calculations can be found in \cite{MC16} (cf. (73)). Thus, \eqref{convrest} gives us that,
\begin{align*}
	\partial_t\tilde J_N \to \partial_tJ_\bfeta = J_\bfeta(\nabla^\bfeta\cdot\bw_\bfeta),\quad \text{ in } L^2(0,T;L^3(\sO)).
\end{align*}
\fi

	Next, we will prove strong convergence of the fluid and structure velocities. First, we obtain the following uniform bounds for the fluid and structure velocities in the Nikoski space $N^{\frac18,2}(0,T_0;\bL^2(\sO)\times \bL^2(\Gamma))$. See \eqref{Nikolski} to recall the definition of Nikolski spaces.

	\begin{lem}\label{compact}
		The sequence of approximate solutions $(\bu_N,\bv_N)$ is bounded uniformly in the Nikolski space $N^{\frac18,2}(0,T_0;\bL^2(\sO)\times \bL^2(\Gamma))$.
	\end{lem}
	\begin{proof}
The proof of this Lemma relies on the steps in the proof of Theorem \ref{thm:Lip2}. 
Namely, our aim is to prove that for any $h>0$
\begin{equation}\label{est_comp}
\int_h^{T_0} \|\tau_h\bu_N-\bu_N\|^2_{\bL^2(\sO)}+\|\tau_h\bv_N-\bv_N\|^2_{\bL^2(\Gamma)}=	(\Delta t)\sum_{n=j}^N \|\bu^n-\bu^{n-j}\|^2_{\bL^2(\sO)} + \|\bv^n-\bv^{n-j}\|^2_{\bL^{2}(\Gamma)} \leq Ch^{\frac18},
\end{equation}
where the constant $C>0$ is independent of $h$ and $N=\frac{T_0}{\Delta t}$. 
To prove this estimate we would like to use the monolithic approximate weak formulation \eqref{weakapprox}, and replace the test functions with the appropriate solutions. However, we  have to be careful since $\bu_N,\bv_N$ are not the 
solutions of the monolithic approximate weak formulation \eqref{weakapprox}, as they satisfy the corresponding subproblems obtained using the Lie splitting strategy. 

To get around this difficulty, we present here the construction of a suitable pair of test functions for the approximate sub-problems \eqref{first} and \eqref{second},
that are expressed in terms of modifications of $\bu_N,\bv_N$, which can be used to derive estimate \eqref{est_comp}. 
Their construction will  mimic the construction of their continuous-in-time counterparts \eqref{test_f} and \eqref{test_s}, as in the proof of Theorem~\ref{thm:Lip2}. 

Let  $h=j\Delta t+l$ for some $0\leq j \leq N$ and $l<\Delta t$. For simplicity of our presentation we will take $l=0$ and refer the interested reader to \cite{MC19} (see (3.8)-(3.10)) for the treatment of the case $l>0$. 

To construct the appropriate test functions we fix $\Delta t$, i.e., we fix $N$, and consider $n$ and $k$ such that $0\leq n \neq k\leq N$.
The plan is to construct $\bu^k$ and $\bv^k$ (we are dropping the subscript $N$ here) in a way that they can be used as test functions for the equations for $\bu^n$ and $\bv^n$, for some $0\leq n \neq k\leq N$. Due to the fact that we are working on moving domains, this is not trivial since these two solutions are defined on different fluid domains.

We start by defining a discrete version of the function $\bu_M(s,t)$ in \eqref{test_f} as follows:	
		\begin{align*}
			\bu^{k,n}_{M}&:= (J^{n}_{})^{-1}\nabla \bA_{\bfeta^n} \left(J^{k}(\nabla  \bA_{\bfeta^k})^{-1} (\bu^k-\bw^k)\right)
			+\bw^k_M+\left( {b^{k}-b^{k,n}_M}\right) \xi_0\chi\\
			&	-(J^{n})^{-1}\nabla \bA_{\bfeta^n}\sB\Big( \text{div} \left( (J^{k}_{}(\nabla  \bA_{\bfeta^k})^{-1} \bw^k)-J^{n}_{}(\nabla  \bA_{\bfeta^n})^{-1} \bw^k_M\right) 	\\
			&-\left( {b^{k}-b^{k,n}_M}\right) \text{div}((J^n)^{-1}\nabla \bA_{\bfeta^n}\xi_0\chi)\Big),
		\end{align*}	
where $\bw^k$ is the harmonic extension of $\bv^k$ in $\sO$, such that $\bw^k=0$ on $\partial\sO\setminus\Gamma$. Similarly,
for any $0\leq k\leq N$, we use the notation $\bv^k_M=P_M\bv^k$ and denote by $\bw^k_M$ the harmonic extension of $\bv^k_M$ in $\sO$, such that $\bw^k_M=0$ on $\partial\sO\setminus\Gamma$. We recall that $\sB$ is the Bogovski operator 
on the fixed domain $\sO$. We also define the discreet correction terms, $b^{k}=\int_\Gamma(\nabla \bfeta^k \times \bv^k) $ and $b^{k,n}_M=\int_\Gamma (\nabla \bfeta^n \times \bv_M^k) $
 and choose a smooth function $\xi_0$ so that it satisfies
$\int_\Gamma(\nabla \bfeta^n \times \xi_0) =1$ for any $n$.

Similarly, the following function is the time discretized version of $\bv_M(s,t)$ in \eqref{test_s}: $$\bv^{k,n}_{M}:=\bv^k_M-\left( {b^{k}-b^{k,n}_M}\right) \xi_0.$$
Then,
 for any $n \leq N$ we define the following pair of test functions written in terms of $\bu^{k,n}_{M}$ and $\bv_{M}^{k,n}$  (compare with \eqref{test_f} and \eqref{test_s}), 
 \begin{align}	\label{Qn1}
			(\bq_n,\bfpsi_n)=\left( (\Delta t) \sum_{k=n-j+1}^n\bu^{k,n}_{M},\,\, (\Delta t) \sum_{k=n-j+1}^n\bv_{M}^{k,n}\right) .
		\end{align}
Now we use $(\bq_n,\bfpsi_n)$ as test functions in the two subproblems \eqref{first}$_3$ and \eqref{second}$_2$ respectively and then apply $\sum_{n=0}^N$ which yields (cf. \eqref{weakform_reg}),
		\begin{align*}	 	
			&-\sum_{n=0}^N\left( \int_{\sO}\left((J^{n+1}) \bu^{n+1}-(J^{n})\bu^{n}\right) \left( \Delta t \sum_{k=n-j+1}^n\bu^{k,n}_{M}\right)  
			+\int_\Gamma(\bv^{n+1}-\bv^{n} )\left( \Delta t \sum_{k=n-j+1}^n\bv_{M}^{k,n}\right) \right) \\
			&=\frac{-1}2\sum_{n=0}^N\int_\sO \left( J^{n+1}-J^n\right)  \bu^{n+1}\cdot \left( \Delta t \sum_{k=n-j+1}^n\bu^{k,n}_{M}\right) +(\Delta t)\sum_{n=0}^N b^{\bfeta_n}(\bu^{n+1},\bw^n,\left( \Delta t \sum_{k=n-j+1}^n\bu^{k,n}_{M}\right) ) \\
			&
			+2\nu(\Delta t)\sum_{n=0}^N\int_{\sO}(J^n) \bD^{\eta^{n}}(\bu^{n+1})\cdot \bD^{\eta^{n}}\left( \Delta t \sum_{k=n-j+1}^n\bu^{k,n}_{M}\right)  \\
			&+(\Delta t)\sum_{n=0}^N \int_\Gamma \Delta\bfeta^{n+\frac12}\cdot \left( \Delta t \sum_{k=n-j+1}^n\Delta \bv_{M}^{k,n}\right)+\frac1N(\Delta t)\sum_{n=0}^N \int_\Gamma \nabla^3\bfeta^{n+\frac12}\cdot \left( \Delta t \sum_{k=n-j+1}^n\nabla^3 \bv_{M}^{k,n}\right)  \\ &+\gamma\int_\Gamma\Lambda^{1+s}\bv^{n+\frac12}\cdot\Lambda^{1+s}\left( \Delta t \sum_{k=n-j+1}^n\bv^{k,n}_{M}\right)  
			\\		
			&=I_1+...+I_5.
		\end{align*}
		After using summation by parts formula, the two terms on the left hand side of the equation above produce the desired terms
		$\int_h^{T_0} \|\tau_h\bu_N-\bu_N\|^2_{\bL^2(\sO)}+\int_h^{T_0}\|\tau_h\bv_N-\bv_N\|^2_{\bL^2(\Gamma)}$. The rest of the terms including the six terms $I_{1,...,5}$ on the right-hand side of the equation above are treated identically as in the proof of Theorem \ref{thm:Lip2}  (see the bounds obtained for the terms $I_{1,...,6}$ following \eqref{weakform_reg}).
		This completes the proof of Lemma~\ref{compact}.
	\end{proof}
	To utilize this result and obtain strong convergence, up to a subsequence, of the approximate solutions, we intend to use the following variant of the Aubin-Lions theorem (see \cite{Tem95} and \cite{S87}).
	\begin{lem}\label{ALthm}
		Assume that $\mathcal{Y}_0\subset\mathcal{Y}$ are Banach spaces such that $\mathcal{Y}_0$ and $\mathcal{Y}$ are reflexive with compact embedding of $\mathcal{Y}_0$ in $\mathcal{Y}$. 
		Then for any $m>0$, the embedding	$$ 
		L^2(0,T;\mathcal{Y}_0)\cap N^{m,2}(0,T;\mathcal{Y})
		\hookrightarrow L^2(0,T;\mathcal{Y})$$	is compact.
	\end{lem}
	Hence, combining Lemmas \ref{compact} and \ref{ALthm} with $\mathcal{Y}_0=\bH^1(\sO)\times \bH^{1+s}(\Gamma)$ and $\mathcal{Y}=\bL^2(\sO)\times\bL^2(\Gamma)$, we see that the sequence $\{(\bu_N,\bv_N)\}$ is relatively compact in $L^2(0,T_0;\bL^2(\sO)\times\bL^2(\Gamma))$. 
	Therefore, we obtained the following strong convergence result for fluid and structure velocities:
	\begin{proposition}\label{prop:strongbv}
		The sequence 
		\begin{align*}
			\{(\bu_N,\bv_N)\} \to (\bu,\bv) \quad\text{ strongly in } L^2(0,T_0;\bL^2(\sO)\times\bL^2(\Gamma)).
		\end{align*}
	\end{proposition}

This completes the convergence results for the approximate solutions
that are necessary to pass to the limit as $N\to\infty$ in the monolithic weak formulation of approximate solutions and show that the limits satisfy the weak formulation of the continuous problem. However, as with all FSI problems defined on moving domains, for which the continuous solution is approximated by a sequence of time-discretized approximate solutions, before we can pass to the limit we need to take care of the corresponding test functions. Namely, the test functions in the approximate weak formulations given in terms of the fixed domain $\sO$, depend on $N$ because they satisfy the {\emph{transformed}} divergence-free condition, where the gradient operator depends on $N$ via the approximate interface displacement $\bfeta_N$. 
Therefore, before we can pass to the limit as $N\to\infty$ (or equivalently $\Delta t \to 0$) in \eqref{weakapprox} we need to construct appropriate test functions that satisfy certain strong convergence properties, and are dense in the space of approximate and continuous test functions. 
This is the subject of the next section.

\subsection{Construction of test functions}\label{sec:test_functions}
In this section, we will construct a pair of test functions for the approximate weak formulation \eqref{weakapprox} and the corresponding limiting equations
that have certain desired properties to pass to the limit as $N\to\infty$. We begin by considering a test pair $(\bq,\bfpsi)$ for some $\bq\in L^\infty(0,T;\bH^{1}(\sO)) \cap H^1(0,T;\bL^2(\sO))$ such that $\nabla^{\bfeta}\cdot\bq=0$, and $\bfpsi\in L^\infty(0,T;\bH^2(\Gamma))\cap H^1(0,T;\bH^{1+s}(\Gamma))$ that satisfy the 
kinematic coupling condition i.e. $\bq|_\Gamma =\bfpsi$.
Now, we will build a pair of functions $(\bq_{N},\bfpsi_{N})$ that approximates $(\bq,\bfpsi)$ in an appropriate sense and is also a valid test function for the approximate system \eqref{weakapprox}.

 We define the approximate fluid test function $\bq_N$, with the aid of the Piola transformation as done previously in the proof of Theorem \ref{thm:Lip2} in Section \ref{subsec:temp}:
\begin{align}\label{VelTest}
	&\bq_{N}={J^{-1}_{ \bfeta_N}}\nabla \bA_{ \bfeta_N}{J^{}_{ \bfeta }} \nabla  \bA^{-1}_{ \bfeta } \left( \bq-\bfpsi\chi\right) +\bfpsi\chi-\frac{\lambda^{ {\bfeta} _N}-\lambda^{ \bfeta }}{\lambda^{N}_0}(\xi_0\chi)\\
	\nonumber
	&+{J^{-1}_{ \bfeta _N}}\sB\left( \text{div}\left( (J_{ \bfeta }(\nabla \bA_{ \bfeta })^{-1}-J_{ \bfeta_N }(\nabla \bA_{ \bfeta_N })^{-1})\bfpsi\chi-\frac{\lambda^{ {\bfeta} _N}-\lambda^{ \bfeta }}{\lambda^N_0}J_{ \bfeta_N }(\nabla \bA_{ \bfeta_N })^{-1}\xi_0\chi\right) \right),
\end{align}	
and for the structure test function we let,
\begin{align}\label{StVelTest}	
	&\bfpsi_N= \bfpsi-\frac{\lambda^{ {\bfeta} _N}-\lambda^{ {\bfeta} }}{\lambda^N_0}\xi_0.
\end{align}	
For the correction terms, we pick an appropriate $\xi_0 \in \bC_0^\infty(\Gamma)$ such that $\lambda_0^N$ defined below is not 0 for any $N\in \mathbb{N}$, 
$$\lambda_0^N(t)=-\int_\Gamma ({\bf id}+ {\bfeta} _N(t))\times \nabla\xi_0. \quad\text{}$$
We also define the following corrector functions that only depend on time,
$$\lambda^{ {\bfeta} _N}(t)=-\int_\Gamma ({\bf id} + {\bfeta}_N (t))\times \nabla\bfpsi (t),\quad \lambda^{ {\bfeta} _{}}(t)=-\int_\Gamma ({\bf id} + {\bfeta}_{} (t))\times \nabla\bfpsi (t).$$
As earlier, $\chi(r)$ is a smooth function on $\sO$ such that $\chi(1)=1$ and $\chi(0)=0$.
Observe that the properties of the Piola transformation (see e.g. Theorem 1.7 in \cite{C88}), imply that
$$\nabla^{ \bfeta _N} \cdot\bq_N=J^{}_{ \bfeta }J^{-1}_{ \bfeta_N } \nabla^{ \bfeta }\cdot\bq=0, \qquad \text{and}\qquad \bq_N|_\Gamma\cdot\bn_N =\bfpsi_N\cdot\bn_N .
$$
Additionally, thanks to \eqref{boundsA1} and \eqref{ale_conv1} we obtain that
\begin{align}
	\notag&\|	\bq_N-\bq\|_{L^\infty(0,T;\bH^1(\sO))}\\ \notag&\leq \|\bA_{ \bfeta _N}-\bA_{ \bfeta }\|_{L^\infty(0,T;\bW^{2,\frac{3}{1-\delta}}(\sO))}\left( \|\bq\|_{L^\infty(0,T;\bH^1(\sO))}+\|\bfpsi\|_{L^\infty(0,T;\bH^2(\Gamma))}\right) \\
	&\qquad+\|\lambda^{ {\bfeta} _N}-\lambda^{ {\bfeta} }\|_{L^\infty(0,T)}\to 0. \label{convq}
\end{align}

Similarly, since $\lambda^\bfeta$ and $\lambda^{\bfeta_N}$ are constant in space, we readily obtain that
\begin{equation}\begin{split}\label{convpsi}
		\bfpsi_N \rightarrow \bfpsi \quad \text{ in }L^\infty(0,T;\bH^2(\bar\Gamma)).
	\end{split}
\end{equation}
{\cblue{Thus, we have shown the following result:
\begin{proposition}\label{test}
The  approximate fluid velocity test functions $\bq_{N}$ constructed in \eqref{VelTest}, and the approximate structure velocity test functions $\bfpsi_N$ constructed 
in \eqref{StVelTest}, satisfy the following properties:
\begin{itemize}
\item $\nabla^{ \bfeta _N} \cdot\bq_N=0$
\item $\bq_N|_\Gamma\cdot\bn_N =\bfpsi_N\cdot\bn_N$.
\end{itemize}
Furthermore, the following strong convergence results hold:
$$
\bq_N \rightarrow \bq \ \ {\rm in}\ \ {L^\infty(0,T;\bH^1(\sO))}  \ \ {\rm and} \ \ 
\bfpsi_N \rightarrow \bfpsi \quad \text{ in }L^\infty(0,T;\bH^2(\bar\Gamma)),
$$
where $(\bq,\bfpsi )$ are the test functions associated with the continuous problem. 
\end{proposition}

\subsection{Passing to the limit}
We are now in a position to pass to the limit in the semi-discrete formulation \eqref{weakapprox}, as $N\to\infty$. 
We use the test functions constructed above in Proposition~\ref{test} as the test functions in the semi-discrete formulation \eqref{weakapprox}
(these test functions are dense in the space of all test functions for the approximate problems).  Then, use the weak and strong convergence results discussed above for approximate solutions, and pass to the limit as $N\to\infty$ in \eqref{weakapprox} to show that the limits of  approximate subsequences satisfy the weak formulation of the continuous problem stated in Definition
\ref{weakform_steady}. Due to Proposition~\ref{equiv}, Definition~\ref{weakform_steady} and Definition \ref{weakform_moving}  are equivalent, which completes the proof of 
Theorem~\ref{thm:exist}.

}}

\section*{Acknowledgment}
SC acknowledges support from the National Science Foundation grants DMS-240892, DMS-2247000, DMS-2011319. 
BM acknowledges partial support from the Croatian Science Foundation, project number IP-2022-10-2962 and from the Croatia-USA bilateral grant “The mathematical framework for the diffuse interface method applied to coupled problems in fluid dynamics”. KT acknowledges support from the National Science Foundation grant DMS-2407197.

\bibliographystyle{plain}
\bibliography{3DFSI}

\newcommand{\noop}[1]{}
\begin{thebibliography}{10}

\bibitem{BH21}
A.~Behzadan and M.~Holst.
\newblock Multiplication in {S}obolev spaces, revisited.
\newblock {\em Ark. Mat.}, 59(2):275--306, 2021.

\bibitem{bodnar2014fluid}
T.~Bodn{\'a}r, G.~Galdi, and {\v{S}}.~Ne{\v{c}}asov{\'a}.
\newblock {\em Fluid-structure interaction and biomedical applications}.
\newblock Springer, 2014.

\bibitem{Canic21}
S.~{\v C}ani{\'c}.
\newblock Moving boundary problems.
\newblock {\em Bull. Amer. Math. Soc. (N.S.)}, 58(1):79--106, 2021.

\bibitem{CDEG}
A.~Chambolle, B.~Desjardins, M.~J. Esteban, and C.~Grandmont.
\newblock Existence of weak solutions for the unsteady interaction of a viscous
  fluid with an elastic plate.
\newblock {\em J. Math. Fluid Mech.}, 7(3):368--404, 2005.

\bibitem{ChenTriggiani90}
S.~Chen and R.~Triggiani.
\newblock Gevrey class semigroups arising from elastic systems with gentle
  dissipation: the case {$0<\alpha <\frac 12$}.
\newblock {\em Proc. Amer. Math. Soc.}, 110(2):401--415, 1990.

\bibitem{C88}
P.~Ciarlet.
\newblock {\em Mathematical elasticity, Vol. I}.
\newblock volume 20 of Studies in Mathematics and its Applications.
  North-Holland Publishing Co., Amsterdam, 1988.

\bibitem{FGH02}
A.~Fursikov, M.~Gunzburger, and L.~Hou.
\newblock Trace theorems for three-dimensional, time-dependent solenoidal
  vector fields and their applications.
\newblock {\em Trans. Amer. Math. Soc.}, 354(3):1079--1116, 2002.

\bibitem{Galdi}
G.~P. Galdi.
\newblock {\em An introduction to the mathematical theory of the
  {N}avier-{S}tokes equations}.
\newblock Springer Monographs in Mathematics. Springer, New York, second
  edition, 2011.
\newblock Steady-state problems.

\bibitem{GHH06}
M.~Gei\ss~ert, H.~Heck, and M.~Hieber.
\newblock On the equation {${\rm div}\,u=g$} and {B}ogovski\u{\i}'s operator in
  {S}obolev spaces of negative order.
\newblock In {\em Partial differential equations and functional analysis},
  volume 168 of {\em Oper. Theory Adv. Appl.}, pages 113--121. Birkh\"{a}user,
  Basel, 2006.

\bibitem{GT83}
D.~Gilbarg and N.~Trudinger.
\newblock {\em Elliptic partial differential equations of second order}, volume
  224 of {\em Grundlehren der mathematischen Wissenschaften [Fundamental
  Principles of Mathematical Sciences]}.
\newblock Springer-Verlag, Berlin, second edition, 1983.

\bibitem{G08}
C.~Grandmont.
\newblock Existence of weak solutions for the unsteady interaction of a viscous
  fluid with an elastic plate.
\newblock {\em SIAM J. Math. Anal.}, 40(2):716--737, 2008.

\bibitem{GH16}
C.~Grandmont and M.~Hillairet.
\newblock Existence of global strong solutions to a beam-fluid interaction
  system.
\newblock {\em Arch. Ration. Mech. Anal.}, 220(3):1283--1333, 2016.

\bibitem{GrandmOntHillairetLeq19}
C.~Grandmont, M.~Hillairet, and J.~Lequeurre.
\newblock Existence of local strong solutions to fluid-beam and fluid-rod
  interaction systems.
\newblock {\em Ann. Inst. H. Poincar\'e{} C Anal. Non Lin\'eaire},
  36(4):1105--1149, 2019.

\bibitem{grandmont2024existence}
C.~Grandmont and L.~Sabbagh.
\newblock Existence and uniqueness of strong solutions to a bi-dimensional
  fluid-structure interaction system.
\newblock 2024.

\bibitem{G11}
P.~Grisvard.
\newblock {\em Elliptic problems in nonsmooth domains}, volume~69 of {\em
  Classics in Applied Mathematics}.
\newblock Society for Industrial and Applied Mathematics (SIAM), Philadelphia,
  PA, 2011.
\newblock Reprint of the 1985 original [MR0775683], With a foreword by Susanne
  C. Brenner.

\bibitem{kaltenbacher2018mathematical}
B.~Kaltenbacher, I.~Kukavica, I.~Lasiecka, R.~Triggiani, A.~Tuffaha, and
  J.~Webster.
\newblock {\em Mathematical theory of evolutionary fluid-flow structure
  interactions}.
\newblock Springer, 2018.

\bibitem{KSS23}
M.~Kampschulte, S.~Schwarzacher, and G.~Sperone.
\newblock Unrestricted deformations of thin elastic structures interacting with
  fluids.
\newblock {\em J. Math. Pures Appl. (9)}, 173:96--148, 2023.

\bibitem{KT12}
I.~Kukavica and A.~Tuffaha.
\newblock Solutions to a fluid-structure interaction free boundary problem.
\newblock {\em Discrete Contin. Dyn. Syst.}, 32(4):1355--1389, 2012.

\bibitem{FSIHeat}
V.~M\'acha, B.~Muha, \v{S}. Ne\v{c}asov\'a, A.~Roy, and S.~Trifunovi\'c.
\newblock Existence of a weak solution to a nonlinear fluid-structure
  interaction problem with heat exchange.
\newblock {\em Comm. Partial Differential Equations}, 47(8):1591--1635, 2022.

\bibitem{MaityTakhashi21}
D.~Maity and T.~Takahashi.
\newblock {$L^p$} theory for the interaction between the incompressible
  {N}avier-{S}tokes system and a damped plate.
\newblock {\em J. Math. Fluid Mech.}, 23(4):Paper No. 103, 23, 2021.

\bibitem{MS22}
B.~Muha and S.~Schwarzacher.
\newblock Existence and regularity of weak solutions for a fluid interacting
  with a non-linear shell in three dimensions.
\newblock {\em Ann. Inst. H. Poincar\'{e} C Anal. Non Lin\'{e}aire},
  39(6):1369--1412, 2022.

\bibitem{MC13}
B.~Muha and S.~\v{C}ani\'{c}.
\newblock Existence of a weak solution to a nonlinear fluid-structure
  interaction problem modeling the flow of an incompressible, viscous fluid in
  a cylinder with deformable walls.
\newblock {\em Arch. Ration. Mech. Anal.}, 207(3):919--968, 2013.

\bibitem{MC14}
B.~Muha and S.~\v{C}ani\'{c}.
\newblock Existence of a solution to a fluid-multi-layered-structure
  interaction problem.
\newblock {\em J. Differential Equations}, 256(2):658--706, 2014.

\bibitem{SunBorSlip}
B.~Muha and S.~\v{C}ani\'c.
\newblock Existence of a weak solution to a fluid-elastic structure interaction
  problem with the {N}avier slip boundary condition.
\newblock {\em J. Differential Equations}, 260(12):8550--8589, 2016.

\bibitem{MC16}
B.~Muha and S.~\v{C}ani\'{c}.
\newblock Existence of a weak solution to a fluid-elastic structure interaction
  problem with the {N}avier slip boundary condition.
\newblock {\em J. Differential Equations}, 260(12):8550--8589, 2016.

\bibitem{MC19}
B.~Muha and S.~\v{C}ani\'{c}.
\newblock A generalization of the {A}ubin-{L}ions-{S}imon compactness lemma for
  problems on moving domains.
\newblock {\em J. Differential Equations}, 266(12):8370--8418, 2019.

\bibitem{SebastianWSU}
S.~Schwarzacher and M.~Sroczinski.
\newblock Weak-strong uniqueness for an elastic plate interacting with the
  {N}avier-{S}tokes equation.
\newblock {\em SIAM J. Math. Anal.}, 54(4):4104--4138, 2022.

\bibitem{S87}
J.~Simon.
\newblock Compact sets in the space {$L^p(0,T;B)$}.
\newblock {\em Ann. Mat. Pura Appl. (4)}, 146:65--96, 1987.

\bibitem{TaoDisspersive}
T.~Tao.
\newblock {\em Nonlinear dispersive equations}, volume 106 of {\em CBMS
  Regional Conference Series in Mathematics}.
\newblock Conference Board of the Mathematical Sciences, Washington, DC; by the
  American Mathematical Society, Providence, RI, 2006.
\newblock Local and global analysis.

\bibitem{T24slip}
K.~Tawri.
\newblock A stochastic fluid-structure interaction problem with navier slip
  boundary condition.
\newblock {\em SIAM Journal on Mathematical Analysis, to appear}, 2024.
\newblock arXiv:2402.13303.

\bibitem{tebou2021regularity}
L.~Tebou.
\newblock Regularity and stability for a plate model involving fractional
  rotational forces and damping.
\newblock {\em Zeitschrift f{\"u}r angewandte Mathematik und Physik},
  72(4):158, 2021.

\bibitem{Tem95}
R.~Temam.
\newblock {\em Navier-{S}tokes equations and nonlinear functional analysis},
  volume~66 of {\em CBMS-NSF Regional Conference Series in Applied
  Mathematics}.
\newblock Society for Industrial and Applied Mathematics (SIAM), Philadelphia,
  PA, second edition, 1995.

\bibitem{Srdjan20}
S.~Trifunovi\'c and Y.-G. Wang.
\newblock Existence of a weak solution to the fluid-structure interaction
  problem in 3{D}.
\newblock {\em J. Differential Equations}, 268(4):1495--1531, 2020.

\bibitem{MCG20}
S.~\v{C}ani\'{c}, M.~Gali\'{c}, and B.~Muha.
\newblock Analysis of a 3{D} nonlinear, moving boundary problem describing
  fluid-mesh-shell interaction.
\newblock {\em Trans. Amer. Math. Soc.}, 373(9):6621--6681, 2020.

\bibitem{CKMT24}
S.~\v{C}ani\'{c}, J.~Kuan, B.~Muha, and K.~Tawri.
\newblock {\em Deterministic and Stochastic Fluid-Structure Interaction}.
\newblock Advances in Mathematical Fluid Mechanics. Birkh\"{a}user/Springer,
  2024, conditionally accepted.

\bibitem{V12}
I.~Vel\v{c}i\'{c}.
\newblock Nonlinear weakly curved rod by {$\Gamma$}-convergence.
\newblock {\em J. Elasticity}, 108(2):125--150, 2012.

\end{thebibliography}
	
\end{document}